\documentclass{siamltex}
\usepackage{amsfonts}
\usepackage{amsmath}
\usepackage{mathtools}
\usepackage{latexsym}
\usepackage{amssymb}
\usepackage{booktabs}
\usepackage{array}

\usepackage{graphicx,overpic,xcolor}
\usepackage[list=true, 
labelformat=brace, position=top]{subcaption}
\graphicspath{{./pictures/}{./asymptote/}}
\DeclareGraphicsExtensions{.png,.pdf}

\usepackage[english=american]{csquotes}
\usepackage{hyperref}
\hypersetup{
    colorlinks=true, 
    linktoc=all,     
    linkcolor=blue,  
    citecolor=red,
}
\usepackage{algorithm}
\usepackage{algpseudocode}
\usepackage{setspace}
\usepackage{pgfplots}
\pgfplotsset{compat=newest}
\usetikzlibrary{patterns,backgrounds,calc}
\definecolor{rwthBlue}{RGB}{0,84,159}
\definecolor{rwthLightBlue}{RGB}{142,186,229}
\definecolor{rwthOrange}{RGB}{246,168,0}
\definecolor{rwthGreen}{RGB}{87,171,39}
\definecolor{rwthDarkGreen}{RGB}{0,97,101}
\definecolor{rwthLightGreen}{RGB}{189,205,0}
\definecolor{rwthCyan}{RGB}{0,152,161}
\definecolor{rwthRed}{RGB}{204,7,30}
\definecolor{rwthDarkRed}{RGB}{161,16,53}
\definecolor{rwthLila}{RGB}{122,111,172}
\definecolor{rwthPurple}{RGB}{97,33,88}
\definecolor{rwthYellow}{RGB}{255,237,0}
\definecolor{IGPMred}{RGB}{156,12,16}
\definecolor{IGPMblue}{RGB}{0,105,164}

\newcounter{assum}

\newcommand{\bA} {\mathbf{A}}

\newcommand{\bB} {\mathbf{B}}
\newcommand{\bL} {\mathbf{L}}

\newcommand{\bD} {\mathbf{D}}

\newcommand{\bP} {\mathbf{P}}

\newcommand{\R} {\mathbb{R}}
\newcommand{\T} {\mathcal{T}}

\newcommand{\F} {\mathcal{F}}

\newcommand{\bx} {\mathbf{x}}

\newcommand{\by} {\mathbf{y}}

\newcommand{\bn} {\mathbf{n}}

\newtheorem{remark}{Remark}

\newcommand{\jumpleft}{[\![}
\newcommand{\jumpright}{]\!]}
\newcommand{\jump}[1]{\jumpleft #1 \jumpright}

\newcommand{\averageleft}{\{\!\!\{}
\newcommand{\averageright}{\!\}\!\!\}}
\newcommand{\average}[1]{\averageleft \! #1 \averageright}

\newcommand{\spacejump}[1]{\jump{#1}}
\newcommand{\cTG}{\mathcal{T}_h^\Gamma}
\newcommand{\OG}{\Omega_h^\Gamma}

\newcommand{\MAT}[1]{\mathbf{#1}}
\newcommand{\lin}{{\text{\tiny lin}}}

\newcommand{\Gammalin}{\Gamma^{ \lin}}
\newcommand{\Omihex}{{\Omega_{i,h}^{\rm ex}}}
\newcommand{\Omonehex}{{\Omega_{1,h}^{\rm ex}}}
\newcommand{\Omtwohex}{\Omega_{2,h}^{\rm ex}}
\newcommand{\Omihmin}{{\Omega_{i,h}^{-}}}
\newcommand{\Omonehmin}{{\Omega_{1,h}^{-}}}
\newcommand{\Omtwohmin}{\Omega_{2,h}^{-}}
\newcommand{\OmhGamma}{\Omega_h^{\Gamma}}
\newcommand{\VhFD}{V_h^\mathrm{FD}}
\newcommand{\ahFD}{A_h^\mathrm{FD}}
\def\cdt{{\hskip-0.08ex\cdot\hskip-0.08ex}}
\def\b|{{\|\hskip-0.16ex|}}

\title{Analysis of optimal preconditioners for CutFEM}

\author{Sven Gro{\SS}\thanks{Institut f\"ur
Geometrie und Praktische  Mathematik, RWTH Aachen University, D-52056 Aachen,
Germany; email: {\tt gross@igpm.rwth-aachen.de}}
\and Arnold Reusken\thanks{Institut f\"ur
Geometrie und Praktische  Mathematik, RWTH Aachen University, D-52056 Aachen,
Germany; email: {\tt reusken@igpm.rwth-aachen.de}} 
 }

\begin{document}

\maketitle

\begin{abstract}
In this paper we consider a class of unfitted finite element methods for scalar elliptic  problems. These so-called CutFEM methods use standard finite element spaces on a fixed unfitted triangulation combined with the Nitsche technique and a ghost penalty stabilization. As a model problem we consider the application of such a method to the Poisson interface problem. We introduce and analyze a new class of preconditioners  that is based on a subspace decomposition approach. The unfitted finite element space is split into two subspaces, where one subspace is the standard finite element space associated to the background mesh and the second subspace is spanned by all cut basis functions corresponding to nodes on the cut elements. We will show that this splitting is stable, uniformly in the discretization parameter  and in the location of the interface in the triangulation. Based on this we introduce an efficient preconditioner that is uniformly spectrally equivalent to the stiffness matrix. 
Using a similar splitting, it is shown that the same preconditioning approach can also be applied to a fictitious domain CutFEM discretization of the Poisson equation.
Results of numerical experiments are included that illustrate optimality of such preconditioners for the Poisson interface problem and the Poisson fictitious domain problem.
\end{abstract}


\section{Introduction}
In recent years many papers appeared in which the so-called CutFEM paradigm is developed and  analyzed, cf. the overview references \cite{burman2015cutfem,CutFEM}. In this approach, for discretization of a partial differential equation a fixed \emph{unfitted mesh} is used that is not aligned with a (moving) interface and/or a complex domain boundary. On this mesh \emph{standard finite element spaces} are used. For treating the boundary and/or interface conditions, either a Lagrange multiplier technique or Nitsche's method is applied. In the setting of the present paper we restrict to \emph{Nitsche's method}. Furthermore, to avoid extreme ill-conditioning of the resulting discrete systems (due to ``small cuts'') a stabilization technique is used. The most often used approach is the \emph{ghost-penalty stabilization} \cite{Burman2010}. In the literature the different components of this general CutFEM are studied, error analyses are presented and different fields of applications are studied \cite{burman2015cutfem,CutFEM}. Related unfitted finite element methods are  popular in fracture mechanics \cite{fries2010extended}; in that community these are often called extended finite element methods (XFEM).

Almost all papers on CutFEM (or XFEM) either treat  applications of this metho-dology or present discretization error analyses. In relatively very few papers efficient solvers for the resulting discrete problems are studied. 
In \cite{burmanhansbo12,zawakrbe13,HaZa2014}, for the resulting stiffness matrix condition  number, bounds of the form $ch^{-2}$, with $h$ a mesh size parameter  and $c$ a constant that is independent of how an interface or boundary intersects the triangulation, have been derived. In \cite{burmanhansbo12} a fictitious domain variant of CutFEM is introduced and it is shown that discretization of a Poisson equation using this method yields a stiffness matrix with such a condition number bound. In \cite{zawakrbe13} a similar result is derived for CutFEM applied to  a Poisson interface problem. In \cite{HaZa2014}  a condition number bound is derived  for CutFEM applied to a Stokes interface problem. These papers do \emph{not} treat efficient preconditioners for the stiffness matrix. 

There are  few papers in which (multigrid type) efficient preconditioners for CutFEM or closely related discretizations (e.g., XFEM) are treated, e.g.,  \cite{berger2012inexact,badia2017robust,de2017preconditioning,jo2018geometric,de2020preconditioning,ludescher1,PhDLudescher}.
In none of these papers a rigorous analysis of the spectral quality of the preconditioner is presented. The only paper that we know of that contains such a rigorous analysis is \cite{lehrenfeld2017optimal}. In that paper a CutFEM \emph{without} stabilization is analyzed for a \emph{two}-dimensional Poisson interface problem.

The main topic of the present paper is an analysis of a (new) \emph{subspace decomposition preconditioning technique}  for a  CutFEM discretization of  elliptic interface problems  and for a CutFEM  fictitious domain method. These discretization methods are known in the literature and are typical representatives of the CutFEM methodology \cite{burmanhansbo10,burmanhansbo12,Massing2014}. This preconditioning technique leads to very natural and \emph{optimal preconditioners}, in a sense as explained in section~\ref{sectPrecond}.  We expect that similar preconditioners can be developed and rigorously analyzed for other CutFEM applications such as a Stokes fictitious domain method and Stokes interface problems.

  We explain the key idea of the preconditioner for  the interface problem. 
  In the CutFEM applied to such an elliptic interface problem one uses a standard $H^1$-conforming finite element space on a triangulation that is not fitted to the interface. For treating the interface conditions a Nitsche technique is used, leading to additional bilinear forms (consistency and penalty terms) in the variational formulation of the discrete problem. To damp the instabilities due to ``small cuts'' a ghost-penalty stabilization term is also added in the discrete variational formulation. The finite element space used in the CutFEM  has a natural  splitting into two subspaces, a ``global'' and a ``local'' one. The global subspace is spanned by all standard nodal basis functions on the whole triangulation,  and the local space is spanned by  nodal cut basis functions ``close to'' the interface. The precise definition of a ``cut'' basis function is given in Section~\ref{discrA}. We will show that this space splitting is stable, uniformly in the discretization parameter $h$ and in the location of the interface in the triangulation. We also prove that the Galerkin discretization in the local subspace leads (after diagonal scaling) to a uniformly well-conditioned matrix and that the Galerkin discretization in the global subspace  is uniformly equivalent to the standard finite element discretization of the  Poisson interface problem on the global domain.  Using the latter property it follows that a multigrid method yields an optimal preconditioner for the Galerkin discretization in the global subspace.  An additive Schwarz subspace correction method (or, equivalently, block Jacobi) thus yields an optimal preconditioner for the CutFEM  discretization of the interface problem. The same approach applies, with minor modifications, to a CutFEM fictitious domain discretization of scalar elliptic problems. 
  
  In the literature on CutFEM one finds two different presentations of the finite element space that is used.
  One either defines the space as a global finite element space that is enriched by suitable discontinuous local functions or as a space consisting of two overlapping global finite element spaces. In this paper we will use both definitions and explain the reason for this in Section~\ref{discrA}.  
  
  We briefly address relations between the  results in this paper and in \cite{lehrenfeld2017optimal}. In the latter a CutFEM variant  \emph{without}  stabilization is studied and the preconditioner  is based on a subspace splitting that is similar to the one studied in this paper. The rather technical analysis in  \cite{lehrenfeld2017optimal} is restricted to \emph{linear} finite elements and \emph{two-dimensional} problems. In this paper we consider the CutFEM \emph{with stabilization}. It turns out that this allows an elegant, rather simple  and much more general analysis. In particular, the analysis covers two- and three-dimensional problems, arbitrary polynomial degree finite elements and triangulations that are shape regular but not necessarily quasi-uniform. Furthermore, the analysis of this paper can also be applied to related CutFEM discretizations such as, for example, the CutFEM fictitious domain method. 
  A preliminary preprint version of this paper, in which only the preconditioner for the CutFEM fictitious domain method is treated, is \cite{GrossReusken2021}. 

The paper is organized as follows.   In Section~\ref{sec:disc} we describe a CutFEM discretization of elliptic interface problems known from the literature. In Section~\ref{discrA} two related matrix-vector representations of the discrete problem are introduced.  In Section~\ref{Sectnorm} several uniform norm equivalences are derived that are used in Section~\ref{sectSplitting} to prove a stable splitting property.  Based on this stable splitting we propose (optimal) preconditioners in Section~\ref{sectPrecond}. In Section~\ref{sectNumExp} results of numerical experiments with these preconditioners are presented.

\section{CutFEM for interface problems} \label{sec:disc}
We recall a class of CutFEM methods known from the literature \cite{Hansbo02,burman2015cutfem,HaZa2014}. 
On a bounded connected polygonal domain $\Omega \subset \mathbb{R}^d$, $d=2,3$, we consider the following standard model problem for scalar elliptic interface problems:
\begin{equation}\label{LRPAPER:eq:ellmodel}
\begin{split}
- \mathrm{div} (\alpha_i \nabla {u}) &= \, f
\quad \text{in}~~ \Omega_i , ~i=1,2,  \\
\spacejump{- {\alpha} \nabla {u} } \cdot \bn_\Gamma &= \, 0, \quad \spacejump{{u}} = 0 \quad \text{on}~~\Gamma,  \\
u &= 0 \quad \text{ on } \partial \Omega.
 \end{split}
\end{equation}
Here, $f \in L^2(\Omega)$ is a given source term, $\Omega_1 \cup \Omega_2= \Omega $ a non-overlapping partitioning of the domain, $\Gamma = \overline{\Omega}_1 \cap \overline{\Omega}_2$ is the interface,  $\spacejump{\cdot}$ denotes the usual jump operator across  $\Gamma$ and $\bn_\Gamma$ denotes the unit normal at $\Gamma$ pointing from $\Omega_1$ into $\Omega_2$.   
The weak formulation of the problem \eqref{LRPAPER:eq:ellmodel} is as follows: determine $u \in  H_0^1(\Omega)$ such that
\begin{equation} \label{LRPAPER:eq:ellweakform}
 ( \alpha \nabla u, \nabla v )_\Omega= (f, v)_\Omega \quad \text{for all}~~v \in H_0^1(\Omega).
\end{equation}
Here and in the remainder, $(\cdot,\cdot)_\Omega$ denotes the $L^2$ scalar product on $\Omega$.
We assume  that for discretization a family of shape regular simplicial triangulations $\{\T_h\}_{h>0}$ of $\Omega$ is used which are \emph{not} fitted to $\Gamma$. 
Let $\T_h$ denote a simplicial triangulation of $\Omega$ and $V_h$  the corresponding standard  finite element space of continuous piecewise polynomials up to degree $k$ that have zero values on $\partial \Omega$. Note that in order to simplify the notation the polynomial degree $k$ is not made explicit in the notation $V_h$. The set of all simplices that are cut by the interface $\Gamma$ is denoted by $\cTG$ and the domain formed by these simplices is denoted by $\OmhGamma$. The domain formed by all simplices with nonzero intersection with $\Omega_i$ (``extended subdomain'')  is denoted by $\Omihex$, $i=1,2$. Note that $\OmhGamma \subset\Omihex$ holds. In the CutFEM one uses pairs of finite element functions $u_h:=(u_{1,h},u_{2,h}) \in V_{1,h}\times V_{2,h}$ with
\[
  V_{i,h}:=\{ \, (v_h)_{|\Omihex}~|~v_h \in V_h\, \}.
\]
Based on this space we formulate a discretization of \eqref{LRPAPER:eq:ellmodel} using the  Nitsche technique: determine $ u_h = (u_{1,h},u_{2,h}) \in V_{1,h}\times V_{2,h}$ such that
\begin{equation} \label{LRPAPER:Nitsche1}
 A_h(u_h,v_h) := a_h(u_h,v_h) + N_h(u_h,v_h) +G_h(u_h,v_h)=(f, v_h)_\Omega
\end{equation}
for all $v_h=(v_{1,h},v_{2,h}) \in V_{1,h}\times V_{2,h}$, 
with the bilinear forms
\[
\begin{split}
a_h(u_h,v_h) & := \sum_{i=1}^2 (\alpha_i  \nabla u_{i,h} , \nabla v_{i,h})_{\Omega_{i}}, \\
N_h(u_h,v_h) & := N_h^c(u_h,v_h) + N_h^c(v_h,u_h) + N_h^s(u_h,v_h), \\
N_h^c(u_h,v_h) & := (\average{-\alpha \nabla v_h}\cdot \bn_{\Gamma}, \spacejump{u_h})_{\Gamma}, \quad 
N_h^s(u_h,v_h) :=  \bar \alpha  \gamma (h^{-1}\spacejump{u_h}, \spacejump{v_h})_{\Gamma},\\
G_h(u_h,v_h)& :=\beta  \sum_{i=1}^2  \sum_{\ell=1}^k \sum_{F \in \F_{g,i}} h_F^{2 \ell-1}(\jump{\partial_n^\ell u_{i,h}},\jump{\partial_n^\ell v_{i,h}})_{F}.
\end{split}
\]
 Here $\F_{g,i}$ is a suitable subset of faces in $\OmhGamma$.
Furthermore, 
$\bar \alpha$ is a certain averaging of $\alpha_1$ and $\alpha_2$, depending on the choice of $\average{\cdot}$. The jump of the finite element function $u_h$ across  $\Gamma$ is given  by $\spacejump{u_h}=(u_{1,h}-u_{2,h})_{|\Gamma}$.
For the averaging operator  $\average{\cdot}$ there are different possibilities, cf. \cite{Hansbo02,BZ11,ReuskenLehrenfeld2017} or the overview in \cite{PhDLudescher}. 
For the case of linear finite elements optimal discretization error bounds for this method are derived in \cite{Hansbo02}.  For the higher order case, but without the ghost-penalty term $G_h(\cdot,\cdot)$, optimal discretization error bounds are derived in \cite{ReuskenLehrenfeld2017,ReuskenLehrenfeld2018}. These analyses can be extended to the case with the ghost-penalty stabilization. 

 Since we do not assume quasi-uniformity of the triangulation, the scalings with $h$ and  with $h^{-1}$ are element-wise, e.g., $(h^{-1}u,v)_{\Gamma}:=\sum_{T\in \cTG} h_T^{-1}(u,v)_{T\cap \Gamma}$.
The parameters $\gamma >0$, $\beta >0$ are fixed. The bilinear form $G_h(\cdot,\cdt)$ is the ghost penalty stabilization. Different equivalent variants of this stabilization are known in the literature, cf. \cite{Burman2010,Preuss2018,LehrenfeldOlshanskii2019}. The  choice of a particular variant of this stabilization is not relevant for the analysis in this paper. 
\begin{remark}  \rm
The cut simplices $T \cap \Omega_i$ and the interface segments $T \cap \Gamma$ can have fairly general geometric shapes. This makes it  difficult to develop an efficient quadrature for the computation of integrals over $T \cap \Omega_i$ or $T \cap \Gamma$. For linear finite elements ($k=1$) one usually replaces $\Gamma$ by a suitable piecewise linear approximation $\Gamma_h$, which then results in simple geometric shapes. For higher order finite elements the isoparametric approach introduced in \cite{LehrenfeldHO} can be used. In that approach one assumes that the interface is represented as the zero level of a level set function.
The fundamental idea is the introduction of a (level set function based) parametric mapping $\Theta_h$ of the underlying mesh from a geometrical reference configuration to a final configuration, cf. Fig. \ref{LRPAPER:fig:idea}. We refer to  \cite{LehrenfeldHO} for the definition of $\Theta_h$.
\begin{figure}[ht!]
   \centering
  \includegraphics[width=7.3cm]{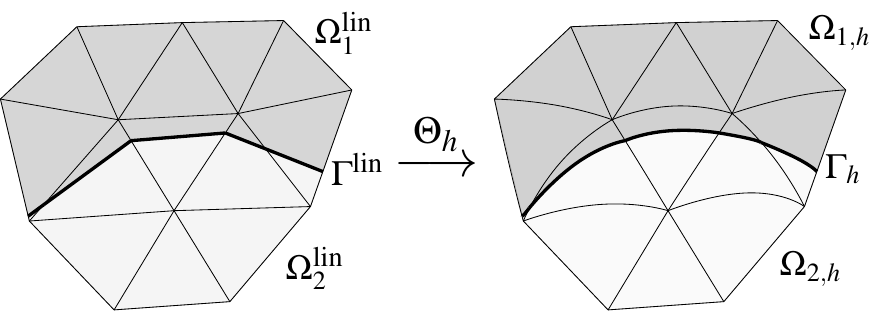} 
  \caption{Basic idea of the isoparametric CutFEM in \cite{LehrenfeldHO}: The piecewise linear approximation  $\Gamma^{\text{lin}}$   is mapped to a higher order approximation $\Gamma_h$  using a  \emph{mesh} transformation $\Theta_h$.}
  \vspace*{-0.4cm}
  \label{LRPAPER:fig:idea} 
\end{figure}
The discretization approach consists of two steps. First, a (higher order) finite element space is considered with respect to the reference configuration. Then the transformation $\Theta_h$ is applied to this space and to the geometries in the variational formulation,  resulting in a new unfitted finite element discretization with an accurate treatment of the geometry. The mapping renders the finite element spaces into \emph{isoparametric finite element spaces}.
The mapping $\Theta_h$ and corresponding quadrature rules are implemented in the add-on library {\sc ngsxfem} \cite{ngsxfem} to Netgen/NGSolve \cite{ngsolve}.
The isoparametric Nitsche unfitted FEM  is a transformed version of the original Nitsche unfitted FE discretization \cite{Hansbo02} with respect to the interface approximation $\Gamma_h = \Theta_h(\Gammalin)$, where $ \Gammalin$ is the zero level of a piecewise linear interpolation of a sufficiently accurate higher order finite element approximation of the level set function $\phi$.  
In the isoparametric approach one uses the spaces $
 V_{i,h}^\Theta:= \{\, v_h \circ \Theta_h^{-1}~|~ v_h \in V_{i,h}\, \}$, $i=1,2$. For further explanation of this method and its discretization error analysis we refer to \cite{ReuskenLehrenfeld2017,ReuskenLehrenfeld2018}. We will not consider this ``perturbation'' due to the isoparametric transformation because it makes the presentation of the analysis below less transparent. We restrict to the method with exact geometry approximation as defined in \eqref{LRPAPER:Nitsche1} since this ``geometric error'' does not play an essential role with respect to the spectral accuracy of the preconditioner introduced in this paper.
 \end{remark}
 
 \bigskip
 It turns out that the preconditioner that we treat in this paper can easily be modified for application to a CutFEM applied in a fictitious domain approach. To explain this more precisely, we describe in the remark below a Nitsche fictitious domain discretization known from the literature. The corresponding preconditioner for this problem is discussed in 
Remark~\ref{RemFDpc}. 
\begin{remark} \label{RemFD1}\rm 
Instead of the interface problem \eqref{LRPAPER:eq:ellmodel} we consider   the Poisson equation
\begin{equation} \label{eq:Poisson}
\begin{split}
 -\Delta u&= f \qquad\text{in }\Omega_1,\\
 u&= g \qquad \text{on } \Gamma=\partial\Omega_1.
\end{split}
\end{equation}
For discretization we apply a fictitious domain method known from the literature \cite{burmanhansbo12,Massing2014}:
determine $u_h \in \VhFD:=V_{1,h}$ such that
\begin{equation} \label{eq:discreteFD}
 \ahFD(u_h,v_h)= (f,v_h)_\Omega - (g,\bn_\Gamma\cdot \nabla v_h)_\Gamma + \gamma(h^{-1} g, v_h)_\Gamma \quad \text{for all}~v_h \in \VhFD,
\end{equation}
where the bilinear form is defined by
\begin{equation} \label{defah} \begin{split}
  \ahFD(u,v)& := (\nabla u, \nabla v)_\Omega - (\bn_\Gamma\cdot \nabla u, v)_\Gamma - (u, \bn_\Gamma\cdot \nabla v)_\Gamma +
   \gamma (h^{-1} u,v)_\Gamma  \\ & + \beta \sum_{\ell=1}^k \sum_{F \in \F_{g,1}} h_F^{2 \ell-1}(\jump{\partial_n^\ell u},\jump{\partial_n^\ell v})_{F}.
   \end{split}
\end{equation}
Here the Nitsche method is used to satisfy (approximately) the boundary condition $u=g$ on $\Gamma$, whereas for the interface problem the Nitsche method is used to enforce the interface condition $\jump{u}=0$ on $\Gamma$. The discretization of the interface problem can be seen as a fictitious domain discretization ``from both sides'' $\Omega_i$, $i=1,2$, with a coupling condition $u_1|_\Gamma = u_2|_\Gamma$ on $\Gamma$.
\end{remark}

 \section{Discrete problems in matrix vector formulation} \label{discrA}

 \tikzset{
     inner vert/.style={circle,draw=black,fill=cyan!40},
     std inner vert/.style={circle,draw=black,fill=white},
     outer vert/.style={rectangle,draw=black,fill=red!40},
     std outer vert/.style={rectangle,draw=black,fill=white},
     inner triang/.style={thick,draw=black,fill=red!10},
     outer triang/.style={thick,draw=black,fill=cyan!10},
     if triang/.style={thick,draw=black,fill=green!10},
     val/.style={circle,fill=black,minimum size=5pt,inner sep=0,outer sep=0}
}
 
 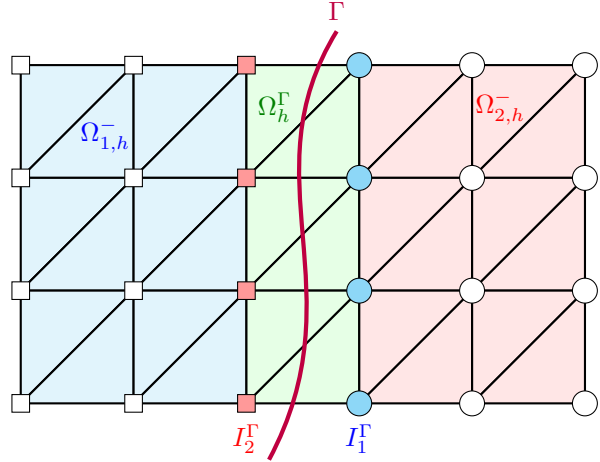
\begin{figure}
  \newcommand{\triang}[4]{\filldraw[#4] (#1.center) -- (#2.center) -- (#3.center) -- cycle; }
  \begin{center}
  \begin{tikzpicture}
    [scale=1.5]
    \foreach \i in {1,2,3,4} {
      \node (z\i) at (0,6-\i) [std outer vert] {};
      \node (a\i) at (1,6-\i) [std outer vert] {};
      \node (b\i) at (2,6-\i) [outer vert] {};
      \node (c\i) at (3,6-\i) [inner vert] {};
      \node (d\i) at (4,6-\i) [std inner vert] {};
      \node (e\i) at (5,6-\i) [std inner vert] {};
    }
    
    \begin{scope}[on background layer]
      \def\outernodes{z2/a2/b2, z3/a3/b3, z4/a4/b4};
      \foreach \na / \nb / \nc
                         [remember=\na as \nalast (initially z1)
                         ,remember=\nb as \nblast (initially a1)
                         ,remember=\nc as \nclast (initially b1)] 
                         in \outernodes 
      {
        \triang{\nalast}{\nblast}{\na}{outer triang}
        \triang{\nblast}{\nb}{\na}{outer triang}
        \triang{\nblast}{\nclast}{\nb}{outer triang}
        \triang{\nclast}{\nc}{\nb}{outer triang}
      }
      \def\ifnodes{b2/c2, b3/c3, b4/c4};
      \foreach \na / \nb [remember=\na as \nalast (initially b1)
                  ,remember=\nb as \nblast (initially c1)] 
                  in \ifnodes 
      {
        \triang{\nalast}{\nblast}{\na}{if triang}
        \triang{\nblast}{\nb}{\na}{if triang}
      }
      \def\innernodes{c2/d2/e2, c3/d3/e3, c4/d4/e4};
      \foreach \na / \nb / \nc 
                         [remember=\na as \nalast (initially c1)
                         ,remember=\nb as \nblast (initially d1) 
                         ,remember=\nc as \nclast (initially e1)] 
                         in \innernodes 
      {
        \triang{\nalast}{\nblast}{\na}{inner triang}
        \triang{\nblast}{\nb}{\na}{inner triang}
        \triang{\nblast}{\nclast}{\nb}{inner triang}
        \triang{\nclast}{\nc}{\nb}{inner triang}
      }
    \end{scope}

    \draw[ultra thick,purple] ($ (b4)!0.5!(c4) + (-0.3,-0.5) $) .. controls (c3.center) and (b2.center) .. ($ (c1) + (-0.2,0.3) $)
       node[purple,above]  {$\Gamma$};
    
    \node[blue] at (barycentric cs:a1=0.3,a2=0.3,z2=0.2) {$\Omega_{1,h}^-$};
    \node[green!50!black] at (barycentric cs:b1=0.3,b2=0.3,c1=0.2) {$\Omega_h^\Gamma$};
    \node[red] at (barycentric cs:d1=0.3,d2=0.3,e1=0.2) {$\Omega_{2,h}^-$};
    \node[red,below=1ex] at (b4) {$I_2^\Gamma$};
    \node[blue,below=1ex] at (c4) {$I_1^\Gamma$};
  \end{tikzpicture}
  \end{center}
  \caption{Sketch of interface $\Gamma$ and (part of) triangulation $\T_h$ of $\Omega_h$ with interface nodes $I_1^\Gamma$ (blue circles) and $I_2^\Gamma$ (red rectangles). All rectangular nodes form the set $I_1\setminus I_1^\Gamma$, all circular nodes form the set $I_2\setminus I_2^\Gamma$.}
  \label{fig:index}
\end{figure}

\begin{figure}
  \begin{center}
  \begin{tikzpicture}
    [scale=1.5]
  \draw[black] (0,0) node[std outer vert] {}
      -- ++(1,0) node[std outer vert] {}
      -- ++(1,0) node[outer vert] {}
      -- ++(1,0) node[inner vert] {}
      -- ++(1,0) node[std inner vert] {}
      -- ++(1,0) node[std inner vert] {};
      
  \def\uOut{0.2,-0.3,0.5};
  \def\uIn{0.6,0.3,-0.1};
  \def\uh{0.2,-0.3,0.2,0.3,-0.1};
  
  \draw[ultra thick,blue] (0,1) node[val,label=$u_{1,h}\in V_{1,h}$] {}
    \foreach \dy in \uOut {
      -- ++(1,\dy) node[val] {}
    };
  \draw[very thick,red] (2,0.5) node[val] {}
    \foreach \dy in \uIn {
      -- ++(1,\dy) node[val] {}
    } node[label=$u_{2,h}\in V_{2,h}$] {};
  \draw[purple,thick,dashed] (2.4,-0.25) node[below] {$\Gamma$} -- ++(0,3.5);
  
  \draw (0,3) node[left] {0}  -- ++(5,0);
  
  \draw[ultra thick,black] (0,4) node[val,label=$u_0\in V_h$] {}
    \foreach \dy in \uh {
      -- ++(1,\dy) node[val] {}
    };
    
  \draw[ultra thick,blue] (2,3) node[val,label={[label distance=0.5ex]$Q_1 u^\Gamma\in V^\Gamma_1$}]{} 
    -- +(1,0.3) node[val]{};
  \draw[ultra thick,red] (3,3) node[val,label={[label distance=1ex]below:$Q_2 u^\Gamma\in V^\Gamma_2$}]{}
    -- +(-1,-0.4) node[val]{};

  \end{tikzpicture}
  \end{center}
  \caption{1D illustration of transformation $L(u_0,u^\Gamma)^T= (u_{1,h},u_{2,h})^T$.}
  \label{fig:trafo}
\end{figure}
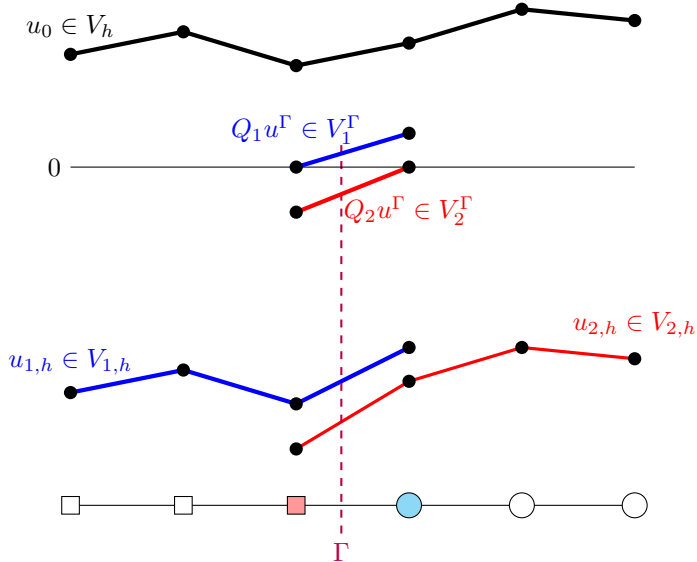

In this section we introduce two matrix vector formulations of the discretization \eqref{LRPAPER:Nitsche1}.
The reason why we use \emph{two} formulations is the following. The space $V_{1,h} \times V_{2,h}$  used in \eqref{LRPAPER:Nitsche1} is a natural one from the point of view of discretization. This space, however, does not have an obvious splitting that is useful for the development of an efficient preconditioner. Below we introduce another space that is the product of the standard global space $V_h$ and a ``local'' space that contains (possibly discontinuous) functions with  supports only on $\Omega_h^\Gamma$. As we will show, this other space has an obvious stable  splitting which then implies an efficient precondioner. These two spaces result in two different matrix-vector representations of the discretization. The relation between these two is  discussed in Remark~\ref{Remneu}.

 We introduce the (closed) subdomains formed by all simplices that are completely contained in $\overline{\Omega}_i$, i.e. $\Omihmin:= \cup \{\, T \in \T_h~|~ T \subset \overline{\Omega}_i\,\}$. Note that $\Omihmin \cap {\rm int}(\OmhGamma) = \emptyset$ and $\Omihmin \cup \OmhGamma = \Omihex$.
 The finite element nodal basis functions of  $V_h$ are denoted by $\phi_j$, $j\in I_0$, for a suitable index set $I_0$. 
 The finite element nodes that are in $\Omihex$ are labeled by $j \in I_i \subset I_0$, $i=1,2$. Let $I^\Gamma\subset I_0$ be the subset of labels corresponding to finite element nodes in $\OmhGamma$ and  $I^\Gamma_i:=\{ j\in I^\Gamma: \text{node $j$ is not in }\Omega_i\}$, $i=1,2$, cf. Figure~\ref{fig:index}. To simplify the presentation, we assume that there are no nodes on $\partial \Omega_i$. Note that  $I^\Gamma= I^\Gamma_1\cup I^\Gamma_2$ and $I_0=(I_1 \setminus I^{\Gamma}_1)\cup (I_2 \setminus I^{\Gamma}_2)$ form  disjoint partitions.
 For finite element nodes $j\in I^\Gamma$ we denote the \emph{cut} basis functions by $\phi_j^\Gamma:=\phi_j|_{\OmhGamma}$. Note that these cut basis functions may be discontinuous across element faces contained in $\partial \OmhGamma$ but are smooth inside all elements. Hence, for  $k \geq 2$ and interior nodes  (i.e., nodes strictly inside an element $T$) we have  $\phi_j^\Gamma= \phi_j$. 
 A natural basis of the finite element space $V_{1,h} \times V_{2,h}$ is given by
 \begin{equation} \label{basis1}
  \big( \{\phi_j\}_{j\in I_1 \setminus I^{\Gamma}_1} \cup \{\phi_j^{\Gamma}\}_{j
  \in I_1^{\Gamma}}\big) \times 
  \big( \{\phi_j\}_{j\in I_2 \setminus I^{\Gamma}_2} 
  \cup \{\phi_j^{\Gamma}\}_{j \in I_2^{\Gamma}}\big).
 \end{equation}
 Using this basis one obtains a matrix-vector representation of \eqref{LRPAPER:Nitsche1} that is denoted by $\bA \bx= \mathbf{b}$.
For preconditioning it is convenient to use another representation of the discrete solution, namely as a suitable \emph{global} finite element function in the space $V_h$ that is \emph{corrected} using a finite element function with support only on $\OmhGamma$. This representation is more in the spirit of the \emph{extended} finite element method.

More precisely, we introduce  the local spaces $ V_i^{\Gamma}:= {\rm span}\{ \big(\phi_j^\Gamma\big)_{j \in I_i^\Gamma}\,\}, ~i=1,2$,  and the product space $ V_h \times V_h^{\Gamma}$ with
\begin{equation} \label{locglob} 
 V_h:={\rm span}\{ \big(\phi_j\big)_{j \in I_0}\,\}, \quad 
  \quad V_h^\Gamma:=V_1^\Gamma\oplus V_2^\Gamma = {\rm span}\{\big(\phi_j^\Gamma\big)_{j \in I^\Gamma}\,\}.
\end{equation}
The bases used in \eqref{basis1} and \eqref{locglob} are the same, hence $ V_h \times V_h^{\Gamma} \simeq V_{1,h}\times V_{2,h}$ holds. We introduce the projections
\[
 Q_i: V_h^\Gamma \to V_i^\Gamma, \quad u^\Gamma=\sum_{j\in I^\Gamma}\beta_j \phi_j^\Gamma \mapsto Q_i u^\Gamma:=\sum_{j\in I_i^\Gamma}\beta_j \phi_j^\Gamma, \quad i=1,2.
\]
With the compact notation $\hat u_h:=(u_0, u^\Gamma)^T\in V_h \times V^{\Gamma}$, a useful isomorphism $L: \,V_h \times V_h^{\Gamma} \to V_{1,h}\times V_{2,h}$ is given by
\begin{equation} \label{nothatu}
 L\hat u_h= L\begin{pmatrix} u_0 \\ u^\Gamma \end{pmatrix} := \begin{pmatrix}
                                                                     {u_0}_{|\Omonehex} + Q_1 u^\Gamma \\
                                                                     {u_0}_{|\Omtwohex} + Q_2 u^\Gamma 
                                                                    \end{pmatrix} =:\begin{pmatrix}
                                                                     u_{1,h} \\
                                                                     u_{2,h} 
                                                                    \end{pmatrix},
\end{equation}
cf. Figure~\ref{fig:trafo}, or in basis notation 
\begin{equation} \label{isomor}
 L\begin{pmatrix} \sum_{j \in I_0} \alpha_j \phi_j \\
    \sum_{j \in I^\Gamma} \beta_j \phi_j^\Gamma
  \end{pmatrix}
= \begin{pmatrix} \sum_{j \in I_1\setminus I_1^\Gamma} \alpha_j \phi_j + \sum_{j\in I_1^\Gamma} (\alpha_j + \beta_j) \phi_j^\Gamma \\
  \sum_{j \in I_2\setminus I_2^\Gamma} \alpha_j \phi_j + \sum_{j\in I_2^\Gamma} (\alpha_j + \beta_j) \phi_j^\Gamma 
\end{pmatrix}.
\end{equation}
The discrete problem can be reformulated in the space   $V_h \times V_h^{\Gamma}$ as follows: determine $\hat u_h \in V_h \times V_h^{\Gamma}$ such that
\begin{equation} \label{disctransformed}
  \hat A_h(\hat u_h,\hat v_h):= A_h(L\hat u_h, L\hat v_h) = f(L\hat v_h) ~~\forall~ \hat v_h=(v_0, v^\Gamma)\in V_h \times V_h^{\Gamma}.  
\end{equation}
The corresponding matrix vector problem is denoted by 
\begin{equation} \label{Axisb}
\hat \bA \hat \bx =\hat{\mathbf{b}}. 
\end{equation}
In the remainder we introduce and analyze a preconditioner for this discrete problem.
\begin{remark} \label{Remneu} \rm 
The matrix representation $\bL$ of the isomorphism \eqref{isomor} is simple, as for the part corresponding to $\phi_j, j\in I_0\setminus I^\Gamma$ and $\phi_j^\Gamma, j\in I^\Gamma$ it is only a permutation matrix, and for the remaining part ($\phi_j, j\in I^\Gamma$) there are exactly two non-zero entries per column. To understand the latter, consider an index $j\in I_2^\Gamma$. Then also $j\in I_1\setminus I_1^\Gamma$ and, hence, $L(\phi_j,0)^T = (\phi_j,\phi_j^\Gamma)^T$. 
The matrices  $\bA$ and $\hat \bA$ are related by $\hat \bA= \bL^T  \bA \bL$. Based on the bilinear form in \eqref{LRPAPER:Nitsche1} the stiffness matrix $\hat \bA$ is easily determined based on the relation $\hat \bA_{m,l}=A_h(L \psi_m,L\psi_l)$, $\psi_m,\psi_l \in \big\{
\phi_j\big\}_{j \in I_0} \cup   \big\{\phi_j^\Gamma\big\}_{j \in I^\Gamma} $. Solving the linear system \eqref{Axisb} approximately using a preconditioner for $\hat \bA$  results in (an approximation of) the discrete solution $\hat u_h$ of \eqref{disctransformed}. We obtain the solution of \eqref{LRPAPER:Nitsche1} using $u_h = L\hat u_h$.
\end{remark}

\begin{remark} \label{RemFD2}\rm 
 In case of the fictitious domain discretization \eqref{eq:discreteFD} of the Poisson problem \eqref{eq:Poisson} on $\Omega_1$, an analogous splitting of the fictious finite element space $V_h^{\rm FD}$ is given by  $V_h^{\rm FD}=V_{h,1}=V_1^-\times V_1^\Gamma$ with $V_1^-:=\mathrm{span}\{\phi_j\}_{j\in I_1\setminus I_1^\Gamma}$. So for $\hat u_{1,h}:=(u_1^-,u_1^\Gamma)^T \in V_1^- \times V_1^\Gamma$ the corresponding isomorphism has the simple form
 \begin{align}
   L_1 \hat u_{1,h}
   = L_1 \begin{pmatrix} u_1^- \\ u_1^\Gamma \end{pmatrix}
   := u_1^- + u_1^\Gamma =: u_{1,h} \in \VhFD, \label{eq:trafoFD} \\[1ex]
   L_1 \begin{pmatrix} \sum_{j \in I_1\setminus I_1^\Gamma} \alpha_j \phi_j \\
    \sum_{j \in I_1^\Gamma} \beta_j \phi_j^\Gamma
  \end{pmatrix}
  = \sum_{j \in I_1\setminus I_1^\Gamma} \alpha_j \phi_j + \sum_{j\in I_1^\Gamma} \beta_j \phi_j^\Gamma. \label{eq:isomorFD}
 \end{align}
 Note that elements from the global space $V_1^- $ are zero on  $\partial\Omega_{1,h}^\mathrm{ex}$.

\end{remark}

\section{Fundamental norm equivalences} \label{Sectnorm}
In this section preliminary results are presented that are  used to derive a new spectral  equivalence result for the bilinear form $\hat A_h(\cdot,\cdot)$  in the main theorem~\ref{thmmain1} below.
In that theorem we essentially show that the splitting of the finite element space $V_h \times V_h^{\Gamma}$ into the subspaces $V_h \times \{0\}$ and  $\{0\} \times V_h^{\Gamma}$ is stable. Applying standard subspace decomposition results, this means that the block diagonal parts of $\hat \bA$ corresponding to these two subspaces constitute a spectrally equivalent approximation of $\hat \bA$. This then leads to the optimal preconditioners introduced in Section~\ref{sectPrecond}. Note that elements from $V_h^{\Gamma}$ have support only on $\Omega_h^\Gamma$. Below we derive several norm equivalences on $\Omega_h^\Gamma$ needed for deriving the stable splitting property. 

For the stability of the subspace splitting we have to analyze the angle in the $\hat A_h(\cdot,\cdot)$ scalar product between the two subspaces. It is convenient to replace this scalar product by one with a simpler structure.  This can be done based on a norm equivalence result known from the literature. 

We use the notation $\sim$ to denote estimates in both directions with constants that are independent of $h$ and of the location of the interface $\Gamma$ in the triangulation. We recall the notation introduced above: for $\hat u_h=(u_0,u^\Gamma)\in V_h \times V^{\Gamma}$ we define $(u_{1,h},u_{2,h}):= ({u_0}_{|\Omonehex} + Q_1u^\Gamma,
                                                                     {u_0}_{|\Omtwohex} + Q_2u^\Gamma) \in V_{1,h}\times V_{2,h}$. 
 From the literature on discretization error analyses of CutFEM, e.g. \cite{Massing2014}, the following fundamental norm equivalence is known:
 \begin{equation} \label{normeq1} \begin{split}
  \hat A_h(\hat u_h,\hat u_h) & \sim \sum_{i=1}^2 \|\nabla u_{i,h}\|_{\Omihex}^2 + \|h^{-\frac12}(u_{1,h}-u_{2,h})\|_\Gamma^2\\
   & =\sum_{i=1}^2 \|\nabla (u_0+ Q_i u^\Gamma)\|_{\Omihex}^2+ \|h^{-\frac12}(Q_1 u^\Gamma- Q_2 u^\Gamma)\|_\Gamma^2,
\end{split}  \end{equation}
for all $\hat u_h= (u_0,u^\Gamma) \in V_h \times V^{\Gamma}$. 
For this uniform norm equivalence to hold \emph{it is essential that a ghost penalty type stabilization is added}.  We derive preliminaries in the following lemmas.  We will use the trace inequality \cite{Hansbo02}:
 \begin{equation}\label{trace2}
\|v\|_{T \cap \Gamma}  \lesssim (h_T^{-\frac12} \|v\|_T + h_T^\frac12 \|\nabla v\|_T), \quad v \in H^1(T).
 \end{equation}
For a subdomain $\omega \subset \Omega$ we use the notation $V_h(\omega):= \{\, (v_h)_{|\omega}~|~ v_h \in V_h\,\}$. 
The result in the next lemma gives a  useful uniform norm equivalence for finite element functions restricted to the local interface strip $\OmhGamma$.
\begin{lemma} \label{lemma1}
The following uniform norm equivalence holds:
\begin{equation} \label{RR1}
  \|h^{-1}v_h\|_{\OmhGamma}^2 \sim  \|h^{-\frac12}v_h\|_\Gamma^2 + \|\nabla v_h\|_{\OmhGamma}^2 \quad \text{for all}~v_h \in V_h(\OG).
\end{equation}
\end{lemma}
\begin{proof}
Using \eqref{trace2} we get
\[
   \|h^{-\frac12}v_h\|_\Gamma^2 = \sum_{T \subset \OG}h_T^{-1} \|v_h\|_{T \cap \Gamma}^2  \lesssim \|h^{-1}v_h\|_{\OmhGamma}^2 + \|\nabla v_h\|_{\OmhGamma}^2. 
\]
Combining this with a standard finite element inverse inequality yields
\[
  \|h^{-\frac12}v_h\|_\Gamma^2 + \|\nabla v_h\|_{\OmhGamma}^2 \lesssim \|h^{-1}v_h\|_{\OmhGamma}^2,
\]
i.e., a uniform estimate in one direction in \eqref{RR1}. We now derive the estimate in the other direction.
We introduce, for $T \in \cTG$, the subdomain consisting of all simplices  in $\cTG$ that have at least a common vertex with $T$, i.e., $\omega_T:= \{\, \tilde T \in \cTG~|~\tilde T \cap T \neq \emptyset \,\}$. Note that due to shape regularity we have $h_{\tilde T} \sim h_T$ for $\tilde T \in \omega_T$ and ${\rm diam}(\omega_T) \sim h_T$. 

 Take $T \in \cTG$, $v_h \in V_h(\OG)$. The area $|T\cap \Gamma |$ can be arbitrary small (``small cuts''), but it follows from \cite[Proposition 4.2]{DemOlsh} that there is an element $\tilde T \in \omega_T$ such that $|\tilde T \cap \Gamma| \geq c_0 h_{\tilde T}^{d-1}$, with  a constant $c_0 >0$ that depends only on shape regularity of $\cTG$ and on smoothness of $\Gamma$. Take such a $\tilde T \in \omega_T$. Take a fixed  $\xi \in  \Gamma \cap  \tilde T$ such  that $|v_h(\xi)|= \max_{x \in \tilde T \cap \Gamma} |v_h(x)| =: \|v_h\|_{\infty,\tilde T \cap \Gamma}$. Take $x \in T $ and let $S$ be a smooth shortest curve in $\omega_T$ that connects $x$ and $\xi$. Due to shape regularity we have $|S| \lesssim h_T$, independent of $x$. This yields
 \[
  v_h(x)= v_h(\xi) + \int_S\frac{\partial v_h}{\partial s} \, ds,
 \]
 with $s$ the arclength parametrization of $S$.
Hence,
\[
  v_h(x)^2 \leq  2 v_h(\xi)^2 + 2 |S|^2 \|\nabla v_h\|_{\infty,\omega_T}^2.
\]
Using integration over $T$, $|T|\sim h_T^d$ and the standard FE norm estimate $\|\nabla v_h\|_{\infty,\omega_T}^2 \lesssim h_T^{-d} \|\nabla v_h\|_{\omega_T}^2$ we get
\begin{equation} \label{est6}
 h_T^{-2} \|v_h\|_T^2 \lesssim h_T^{d-2} \|v_h\|_{\infty,\tilde T \cap \Gamma}^2 + \|\nabla v_h\|_{\omega_T}^2.
\end{equation}
Using $|\tilde T \cap \Gamma| \geq c_0 h_{\tilde T}^{d-1}$ we get
\[
  \|v_h\|_{\infty,\tilde T \cap \Gamma}^2 \lesssim h_{\tilde T}^{1-d} \|v_h\|_{\tilde T \cap \Gamma}^2,
\]
and combining this with the result \eqref{est6} and $h_{\tilde T} \sim h_T$ yields
\[
h_T^{-2} \|v_h\|_T^2 \lesssim h_{\tilde T}^{-1} \|v_h\|_{\tilde T \cap \Gamma}^2 + \|\nabla v_h\|_{\omega_T}^2.
\]
Summing over $ T \in \cTG$ completes the proof.
\end{proof}
 
\begin{remark}\label{Remsimilar}
 \rm  Results similar to \eqref{RR1} are known in the literature. For example, in the papers \cite{burmanembedded,grande2017higher}, for the case of a \emph{quasi-uniform} triangulation the following uniform estimate is derived:
 \begin{equation} \label{compresult}
  \|v_h\|_{\OG} \lesssim h^\frac12 \|v_h\|_{\Gamma} + h \|n \cdot \nabla v_h\|_{\OG}.
 \end{equation}
Note that due to the quasi-uniformity assumption we have a simpler scaling with the global mesh parameter $h$ and that in \eqref{compresult} we  have the \emph{normal} derivative term $\|n \cdot \nabla v_h\|_{\OG}$, with $n$ the normal on $\Gamma$ (constantly extended in the neighborhood $\OG$) instead of the full derivative term $\|\nabla v_h\|_{\OG}$. The proofs of \eqref{compresult} in \cite{burmanembedded,grande2017higher}  are much more involved than the simple proof of Lemma~\ref{lemma1} above. This is due to the fact that in the bound in \eqref{compresult} only the \emph{normal} derivative occurs.
\end{remark}

A second norm equivalence is derived in the following lemma. For this we note that $\partial \OmhGamma$ is the union of two disjoint parts, namely $\partial \OmhGamma \cap \Omonehmin$ and $\partial \OmhGamma \cap \Omtwohmin$. We show that for finite element functions $v_h$ that are zero on one of these two boundary parts the norms $\|\nabla v_h\|_{\OmhGamma}$ and $\|h^{-1} v_h\|_{\OmhGamma}$ are uniformly equivalent.
\begin{lemma} \label{Lemma1}
 The  uniform norm equivalence 
 \begin{equation} \label{normeq}
  \|h^{-1} v_h\|_{\OG} \sim \|\nabla v_h\|_{\OG}
 \end{equation}
 holds for all $ v_h \in V_h(\OmhGamma)$ with ${v_h}_{|\partial \OmhGamma \cap \Omonehmin}=0$ or ${v_h}_{|\partial \OmhGamma \cap \Omtwohmin}=0$.
\end{lemma}
\begin{proof} Take $v_h \in V_h(\OmhGamma)$.
 The estimate in the one direction directly follows from a standard finite element inverse inequality.
Assume that ${v_h}_{|\partial \OmhGamma \cap \Omonehmin}=0$ or ${v_h}_{|\partial \OmhGamma \cap \Omtwohmin}=0$ and take $T \in \cTG$.    By construction $T$ has at least one vertex on $ \OmhGamma \cap \Omonehmin$ and at least one  vertex on $\partial \OmhGamma \cap \Omtwohmin$. Hence, there is vertex of $T$, denoted by $x_\ast$, at which $v_h(x_\ast)=0$ holds.
Let $\hat T$ be the unit simplex and $ F: \hat T \to T$ the affine transformation with $F(0)=x_\ast$. Define $Z:=\{\, p \in \mathcal{P}_k~|~p(0)=0\,\}$ and note that $p \to \|p\|_{\hat T}$ and $p \to \|\nabla p\|_{\hat T}$ define equivalent norms on $Z$. Due to $\hat v_h:=v_h \circ F \in Z$ and this norm equivalence we obtain
\[
  \|v_h\|_T^2 = |T| \|\hat v_h\|_{\hat T}^2 \lesssim |T| \|\nabla \hat v_h\|_{\hat T}^2 \lesssim h_T^2 \|\nabla v_h\|_T^2,
\]
and thus
\[
 \|h^{-1} v_h\|_{\OG}^2 = \sum_{T \in \cTG} h_T^{-2} \|v_h\|_T^2 \lesssim \sum_{T\in \cTG} \|\nabla v_h\|_T^2=\|\nabla v_h\|_{\OG}^2,
\]
which is this estimate in the other direction.
\end{proof}
\ \\
Note that for $v_i^\Gamma \in V_i^\Gamma$ we have ${v_i^\Gamma}_{|\partial \OmhGamma \cap \Omihmin}=0$. Thus we obtain the following corollary.
\begin{corollary} \label{corolmain}
 The following uniform norm equivalence holds
 \begin{equation} \label{normeqCor}
  \|h^{-1} v_i^\Gamma\|_{\OG} \sim \|\nabla v_i^\Gamma\|_{\OG} \quad \text{for all}~  v_i^\Gamma \in V_i^\Gamma,~~i=1,2.
 \end{equation}
\end{corollary}
\ \\[1ex]
Besides this norm equivalence result for finite element functions from the local correction spaces $v_i^\Gamma \in V_i^\Gamma$, $i=1,2$, there also holds a strengthened Cauchy-Schwarz inequality for the two spaces  $V_i^\Gamma$, $i=1,2$. This is shown in Lemma~\ref{lemma4}. For the proof of that lemma it is convenient to use the following elementary estimate.
\begin{lemma} \label{lemma3}
 Let $M \in \mathbb{R}^{m\times m}$ be symmetric positive definite and $\kappa(M):=\|M\|_2\|M^{-1}\|_2$ the spectral condition number. For all $\bx, \by \in \mathbb{R}^m$ with $\langle \bx, \by\rangle = \bx^T \by=0$ the following holds:
 \[
   |\langle M \bx, \by \rangle | \leq \Big(1- \frac{1}{\kappa(M)}\Big) \langle M \bx, \bx \rangle^\frac12 \langle M \by, \by \rangle^\frac12.
 \]
 \end{lemma}
\begin{proof}
 Let $MV=V\Lambda$, with $\Lambda={\rm diag}(\lambda_1, \ldots, \lambda_m)$, $0<\lambda_1 \leq \ldots \leq \lambda_m$, $V^TV=I$ be the orthogonal eigenvector decomposition of $M$. Take $\bx, \by \in \mathbb{R}^m$ with $\langle \bx, \by\rangle=0$ and define $\hat \bx :=V^T \bx$, $\hat \by = V^T \by$. This yields $\langle \hat \bx, \hat \by \rangle =0$, i.e., $\hat x_1 \hat y_1= - \sum_{i=2}^m \hat x_i \hat y_i$. Using this we obtain
 \begin{align*}
 |\langle M \bx, \by \rangle | & =|\langle \Lambda \hat \bx, \hat \by \rangle | = |\sum_{i=1}^m\lambda_i \hat x_i \hat y_i | \\ 
 & =|\sum_{i=2}^m (\lambda_i-\lambda_1)\hat x_i \hat y_i | \leq \max_{2 \leq i \leq m} \frac{\lambda_i-\lambda_1}{\lambda_i} \sum_{i=2}^m \lambda_i |\hat x_i|| \hat y_i |\\
 & \leq \Big( 1- \frac{\lambda_1}{\lambda_m}\Big) \Big(\sum_{i=1}^m \lambda_i \hat x_i^2 \Big)^\frac12 
\Big(\sum_{i=1}^m \lambda_i \hat y_i^2 \Big)^\frac12 \\
 & = \Big( 1- \frac{1}{\kappa(M)}\Big) \langle M \bx, \bx \rangle^\frac12 \langle M \by, \by \rangle^\frac12, 
 \end{align*}
which proves the result.
\end{proof}
\ \\
Using this we obtain the following uniform strengthened Cauchy-Schwarz inequality and a corresponding norm equivalence.
\begin{lemma} \label{lemma4}
 Let $\hat M \in \mathbb{R}^{m \times m}$, $m:= \begin{pmatrix} d+ k \\ k\end{pmatrix}$, be the element mass matrix of $V_h$ on the reference unit simplex $\hat T \subset \mathbb{R}^d$. For $ T\in \cTG$ the estimate
 \begin{equation} \label{est7}
 |(v_1^\Gamma,v_2^\Gamma)_T| \leq \Big(1- \frac{1}{\kappa(\hat M)}\Big) \|v_1^\Gamma\|_T \|v_2^\Gamma\|_T \quad \text{for all}~v_i^\Gamma \in V_i^\Gamma,~i=1,2,
\end{equation}
holds. Furthermore, the uniform norm equivalence
\begin{equation} \label{est7a}
 \|h^{-1}(v_1^\Gamma +v_2^\Gamma)\|_{\OmhGamma} \sim \|h^{-1}v_1^\Gamma\|_{\OmhGamma} +\|h^{-1}v_2^\Gamma\|_{\OmhGamma}, \quad v_i^\Gamma \in V_i^\Gamma, ~i=1,2,
\end{equation}
holds, with constants 1 and $\kappa(\hat M)^{-\frac12}$ in $\sim$.
\end{lemma}
\begin{proof}
 Take $T\in \cTG$, $v_i^\Gamma \in V_i^\Gamma,~i=1,2$.  On $T$ we introduce a local numbering of the element nodal basis functions and choose an ordering such that 
 \[
   (v_1^\Gamma)_{|T}= \sum_{j=1}^{m_0} \beta_j {\phi_j}_{|T}, \quad  (v_2^\Gamma)_{|T}= \sum_{j=m_0+1}^{m} \gamma_j {\phi_j}_{|T}.
 \]
The corresponding coefficient vectors are 
\[ \boldsymbol{\beta}=(\beta_1, \ldots, \beta_{m_0},0,\ldots, 0)^T, \quad \boldsymbol{\gamma}= (0, \ldots, 0,\gamma_{m_0+1}, \ldots, \gamma_m)^T.\]
Note that $\langle \boldsymbol{\beta},\boldsymbol{\gamma} \rangle =0$ holds. Let $M \in \mathbb{R}^{m \times m}$, $M_{i,j}= (\phi_i, \phi_j )_T $, be the element mass matrix. Note that $\kappa (M)=\kappa(\hat M)$ holds. Thus we obtain, using Lemma~\ref{lemma3}:
\begin{align*}
 |(v_1^\Gamma,v_2^\Gamma)_T| & = |\langle M\boldsymbol{\beta},\boldsymbol{\gamma} \rangle| \leq
   \Big(1- \frac{1}{\kappa(M)}\Big) \langle M\boldsymbol{\beta},\boldsymbol{\beta} \rangle^\frac12 
    \langle M\boldsymbol{\gamma},\boldsymbol{\gamma} \rangle^\frac12 \\
    & = \Big(1- \frac{1}{\kappa (\hat M)}\Big) \|v_1^\Gamma\|_T \|v_2^\Gamma\|_T,
\end{align*}
which yields the result \eqref{est7}. Multiplying by $h_T^{-2}$ and summing over $T \in \cTG$ we get $|(h^{-2}v_1^\Gamma,v_2^\Gamma)_{\OmhGamma}| \leq \big(1- \frac{1}{\kappa(\hat M)}\big) \|h^{-1}v_1^\Gamma\|_{\OmhGamma} \|h^{-1}v_2^\Gamma\|_{\OmhGamma}$. This implies
\[
  \|h^{-1}(v_1^\Gamma +v_2^\Gamma)\|_{\OmhGamma}^2 \geq \frac{1}{\kappa (\hat M)} \big(\|h^{-1}v_1^\Gamma\|_{\OmhGamma}^2+ \|h^{-1}v_2^\Gamma\|_{\OmhGamma}^2\big),
\]
which yields the estimate \eqref{est7a} in one direction with constant $\kappa (\hat M)^{-\frac12}$. The estimate in the other direction follows from the triangle inequality. 
\end{proof}

\section{Stable subspace splitting} \label{sectSplitting}
Based on  results from the previous section we now derive a stable splitting result which essentially states that the angles (in the energy scalar product) between the subspaces $V_h$, $V^\Gamma$ in $V_h \times V^\Gamma$ are uniformly bounded away from zero. Based on classical theory cf. \cite{HackbuschIter,Yserentant:93} this then immediately leads
to optimal block-Jacobi type preconditioners. We recall three norm equivalences from the previous section that we need to derive the stable splitting property, namely the ones in \eqref{RR1}, \eqref{normeqCor} and \eqref{est7a}:
\begin{align} \label{Res1}
  \|h^{-1}v_h\|_{\OmhGamma}^2 & \sim  \|h^{-\frac12}v_h\|_\Gamma^2 + \|\nabla v_h\|_{\OmhGamma}^2 \quad \text{for all}~v_h \in V_h(\OG), \\
   \|h^{-1} v_i^\Gamma\|_{\OG} & \sim \|\nabla v_i^\Gamma\|_{\OG} \quad \text{for all}~  v_i^\Gamma \in V_i^\Gamma,~~i=1,2, \label{Res2} \\
    \|h^{-1}(v_1^\Gamma +v_2^\Gamma)\|_{\OmhGamma} &\sim \|h^{-1}v_1^\Gamma\|_{\OmhGamma} +\|h^{-1}v_2^\Gamma\|_{\OmhGamma}, \quad\text{for all}  ~v_i^\Gamma \in V_i^\Gamma, ~i=1,2. \label{Res3}
\end{align}
On $V_h \times V^\Gamma $ we introduce the energy norms
\begin{align*}
  \|\hat u_h\|_a^2 & :=\hat A_h(\hat u_h,\hat u_h), \\ \|\hat u_h\|_b^2& :=\sum_{i=1}^2 \|\nabla u_{i,h}\|_{\Omihex}^2 + \|h^{-\frac12}(u_{1,h}-u_{2,h})\|_\Gamma^2 \\ & =\sum_{i=1}^2 \|\nabla (u_0+ Q_i u^\Gamma)\|_{\Omihex}^2+ \|h^{-\frac12}(Q_1 u^\Gamma- Q_2 u^\Gamma)\|_\Gamma^2,
\end{align*}
with notation as in \eqref{nothatu}. For $\hat u_h=(u_0,u^\Gamma) \in V_h \times V^\Gamma$,  projections $P_i$ on the two subspaces are defined by
\[
  P_0 \hat u_h:=(u_0,0),~~ P_1 \hat u_h:=(0,u^\Gamma).
\]
\begin{lemma} \label{lemA} 
 The following uniform norm equivalence holds
 \begin{equation} \label{E1}
  \|(0,u^\Gamma)\|_b \sim \|h^{-1}u^\Gamma\|_{\OmhGamma} \quad \text{for all}~ u^\Gamma \in V^\Gamma.
 \end{equation}
\end{lemma}
\begin{proof}
 Note that for $u^\Gamma= Q_1 u^\Gamma+Q_1 u^\Gamma$ with $Q_i u^\Gamma \in V_i^\Gamma$ we have 
 \begin{align*} 
 & \|(0,u^\Gamma)\|_b^2  = \sum_{i=1}^2 \|\nabla Q_i u^\Gamma\|_{\OmhGamma}^2+ \|h^{-\frac12}(Q_1 u^\Gamma- Q_2 u^\Gamma)\|_\Gamma^2 \\
 & = \tfrac12 \|\nabla (Q_1 u^\Gamma+Q_2 u^\Gamma)\|_{\OmhGamma}^2+\tfrac12 \|\nabla (Q_1 u^\Gamma-Q_2 u^\Gamma)\|_{\OmhGamma}^2 + \|h^{-\frac12}(Q_1 u^\Gamma- Q_2 u^\Gamma)\|_\Gamma^2\\
  & \overset{\eqref{Res1}}{\sim} \|\nabla (Q_1 u^\Gamma+Q_2 u^\Gamma)\|_{\OmhGamma}^2+ \|h^{-1}(Q_1 u^\Gamma- Q_2 u^\Gamma)\|_{\OmhGamma}^2 \\
  & \overset{\eqref{Res2},\eqref{Res3} }{\sim} \|\nabla (Q_1 u^\Gamma+Q_2 u^\Gamma\|_{\OmhGamma}^2+ \sum_{i=1}^2\|\nabla Q_i u^\Gamma\|_{\OmhGamma}^2 \sim \sum_{i=1}^2\|\nabla Q_i u^\Gamma\|_{\OmhGamma}^2 \\
  &\overset{\eqref{Res2},\eqref{Res3} }{\sim} \|h^{-1}u^\Gamma\|_{\OmhGamma}^2.
 \end{align*}
 Hence the result \eqref{E1} holds.
\end{proof}

\begin{theorem} \label{thmmain1} The following uniform norm equivalences hold:
\begin{align}
  \|\hat u_h\|_b^2 &\sim  \|P_0 \hat u_h\|_b^2 +\|P_1 \hat u_h\|_b^2 , \label{mainb}\\
 \|\hat u_h\|_a^2  & \sim  \|P_0 \hat u_h\|_a^2 +\|P_1 \hat u_h\|_a^2 .\label{maina}
\end{align}
\end{theorem}
\begin{proof}
 The result in \eqref{maina} is a direct consequence of \eqref{mainb} and \eqref{normeq1}. We prove the result \eqref{mainb} as follows: 
 \begin{align*} \|\hat u_h\|_b^2 &  = \sum_{i=1}^2 \|\nabla (u_0+Q_i u^\Gamma)\|_{\Omihex}^2 + \|h^{-\frac12}(Q_1 u^\Gamma-Q_2 u^\Gamma)\|_\Gamma^2 \nonumber \\  
   &= \sum_{i=1}^2 \|\nabla u_0\|_{\Omihmin}^2 +\tfrac12 \|\nabla(2u_0 +Q_1u^\Gamma + Q_2 u^\Gamma)\|_{\OmhGamma}^2 +\tfrac12 \|\nabla(Q_1u^\Gamma - Q_2 u^\Gamma)\|_{\OmhGamma}^2 \\ & \qquad + \|h^{-\frac12}(Q_1 u^\Gamma-Q_2 u^\Gamma)\|_\Gamma^2 \nonumber \\
  & \overset{\eqref{Res1}}{\sim} \sum_{i=1}^2 \|\nabla u_0\|_{\Omihmin}^2 +\|\nabla(2u_0 +Q_1u^\Gamma + Q_2 u^\Gamma)\|_{\OmhGamma}^2 +\|h^{-1}(Q_1u^\Gamma - Q_2 u^\Gamma)\|_{\OmhGamma}^2  \nonumber\\
  & \overset{\eqref{Res3},\eqref{Res2}}{\sim}\sum_{i=1}^2 \|\nabla u_0\|_{\Omihmin}^2 + \|\nabla u_0 +\tfrac12 \nabla(Q_1u^\Gamma + Q_2 u^\Gamma)\|_{\OmhGamma}^2 
  + \sum_{i=1}^2\|\nabla Q_i u^\Gamma\|_{\OmhGamma}^2 \nonumber \\
  & \sim \sum_{i=1}^2 \|\nabla u_0\|_{\Omihmin}^2 + \|\nabla u_0 \|_{\OmhGamma}^2 
  + \sum_{i=1}^2\|\nabla Q_i u^\Gamma\|_{\OmhGamma}^2
  \\
  & \overset{\eqref{Res3},\eqref{Res2}}{\sim}\sum_{i=1}^2 \|\nabla u_0\|_{\Omihex}^2 + \|h^{-1}(Q_1 u^\Gamma+Q_2 u^\Gamma)\|_{\OmhGamma}^2 \nonumber \\
  & \sim \|(u_0,0)\|_b^2 + \|h^{-1} u^\Gamma\|_{\OmhGamma}^2 \sim \|(u_0,0)\|_b^2 + \|(0, u^\Gamma)\|_b^2,
  \end{align*}
  where in the last step we used Lemma~\ref{lemA}. From this and $P_0 \hat u_h=(u_0,0)$, $P_1 \hat u_h=(0,u^\Gamma)$ the result \eqref{mainb} follows. 
\end{proof}

\begin{remark} \rm
 With similar arguments as in   the proof of Theorem~\ref{thmmain1} one can show that  the norm equivalence
 \[
  \|\hat u_h\|_a^2   \sim  \|(u_0,0)\|_a^2 +\|(0,Q_1 u^\Gamma)\|_a^2 +\|(0,Q_2 u^\Gamma)\|_a^2
 \]
holds. Hence, also the splitting of $V_h \times V^\Gamma=V_h \times V_1^\Gamma \times V_2^\Gamma $ in the subspaces  $V_h$, $V_1^\Gamma$ and $V_2^\Gamma$ is stable. However, concerning preconditioning this does not yield significant advantages compared to the stable splitting of $V_h \times V^\Gamma$ in the subspaces $V_h$ and $V^\Gamma$.
\end{remark}

\begin{remark} \label{remrobust}\rm
 The constants in $\sim$ in \eqref{mainb}-\eqref{maina} will depend on the jump in the diffusion coefficient $\alpha$ across the interface. Therefore, the preconditioners proposed in the next section are not expected to be robust with respect to large jumps in this coefficient. We expect that robustness can be obtained using  suitable scalings in \eqref{mainb}-\eqref{maina} that depend on the diffusion coefficient. This will be  analyzed in future work. 
\end{remark}
\ \\

A stable subspace splitting result similar to \eqref{mainb} also holds for the fictitious domain bilinear form with subspaces $V_1^-$ and $V_1^{\Gamma}$, cf. Remarks~\ref{RemFD1} and \ref{RemFD2}. On $V_1^- \times V_1^{\Gamma}$ we define the  energy norms
\begin{align*}
  \|\hat u_{1,h}\|_{a,\mathrm{FD}}^2 & := \ahFD(L_1\hat u_{1,h},L_1\hat u_{1,h}), \\ 
  \|\hat u_{1,h}\|_{b,\mathrm{FD}}^2& := \|\nabla u_{1,h}\|_{\Omega_{1,h}^\mathrm{ex}}^2 + \gamma\|h^{-\frac12}u_{1,h}\|_\Gamma^2 \\ 
  & = \|\nabla (u_1^-+ u_1^\Gamma)\|_{\Omega_{1,h}^\mathrm{ex}}^2 + \gamma\|h^{-\frac12}(u_1^- + u_1^\Gamma)\|_\Gamma^2,
\end{align*}
with notation as in \eqref{eq:trafoFD}. From the literature \cite{burmanhansbo12,Massing2014} we have (for $\gamma$ sufficiently large) the uniform norm equivalence
 \begin{align}
   \ahFD(u_{1,h},u_{1,h}) \sim \|\hat u_{1,h}\|_{b,\mathrm{FD}}^2 \qquad\text{for all } u_{1,h}\in\VhFD.
 \end{align}
Along the same lines as in the proof of \eqref{mainb} with $u_0$ replaced by $u_1^-$, $Q_1 u^\Gamma$ replaced by $u_1^\Gamma$ and $u_{2,h}=Q_2u^\Gamma=0$ one obtains
 for $\hat u_{1,h}=(u_1^-,u_1^\Gamma)\in V_1^-\times V_1^\Gamma$ the  uniform norm equivalence $
   \|\hat u_{1,h}\|_{b,\mathrm{FD}}^2 \sim \|(u_1^-,0)\|_{b,\mathrm{FD}}^2 + \|(0,u_1^\Gamma)\|_{b,\mathrm{FD}}^2$. Thus we get the 
   uniform norm equivalence
   \begin{equation}\label{eq:mainaFD} 
   \|\hat u_{1,h}\|_{a,\mathrm{FD}}^2 \sim \|(u_1^-,0)\|_{a,\mathrm{FD}}^2 + \|(0,u_1^\Gamma)\|_{a,\mathrm{FD}}^2, 
   \end{equation}
which yields the stable subspace splitting result for the fictitious domain method.

\section{Optimal preconditioners} \label{sectPrecond}
We return to the linear system $\hat \bA  \bx =\hat{\mathbf{b}}$ in \eqref{Axisb}. We introduce some notation to represent the subspace splitting in  matrix-vector format. The coefficient vector $\bx$ that represents the unknown finite element function $\hat u_h= (u_0,u^\Gamma)$ is split into the parts corresponding to $u_0$ and $u^\Gamma$, i.e., $\bx=(\bx_0,\bx_1)$ with 
\[
  u_0=\sum_{j\in I_0} x_{0,j} \phi_j,~~ u^\Gamma = \sum_{j \in I^\Gamma} x_{1,j} \phi_{j}^\Gamma.
\]
We define corresponding projections $\bP_i$ by $\bP_0 \bx = (\bx_0,0)$, $\bP_1 \bx = (0,\bx_1)$.
The Galerkin projections on the subspaces are denoted by $\hat \bA_i$, i.e., we have the relations
\[
  \bx_i^T \hat \bA_i \bx_i= \bx^T \bP_i\hat \bA \bP_i \bx = \hat A_h(P_i \hat u_h,P_i \hat u_h)= \|P_i \hat u_h\|_a^2, \quad i=0,1. 
\]
Let $\bD_A:={\rm blockdiag}(\hat \bA_0,\hat \bA_1)$ be the   blockdiagonal matrix corresponding to the Galerkin projections on the subspaces. The result  \eqref{maina}  in matrix formulation yields that $\bD_A$ \emph{is spectrally equivalent to} $\hat \bA$:
\[
 \bx^T \bD_A \bx = \sum_{i=0}^1 \bx_i^T \hat \bA_i \bx_i =\sum_{i=0}^1 \bx^T \bP_i\hat \bA \bP_i \bx = \sum_{i=0}^1\|P_i \hat u_h\|_a^2 \sim \|\hat u_h\|_a^2 = \bx^T \hat \bA \bx.
\]
Hence $\bD_A$ is an \emph{optimal preconditioner} for $\hat \bA$ in the sense that the spectral condition number $\lambda_{\rm max}(\bD_A^{-1} \hat \bA)/\lambda_{\rm min}(\bD_A^{-1} \hat \bA)$ is \emph{uniformly bounded} both with respect to the mesh size $h$ and the location of $\Gamma$ in the triangulation. Note that this condition number may depend on the size of the jumps in the diffusion coefficient $\alpha$, cf. Remark~\ref{remrobust}. 

Clearly the  preconditioner $\bD_A$, which we call the \emph{exact} preconditioner, is not computationally efficient. We now explain how the diagonal blocks $\hat \bA_i$, $i=0,1$, can be replaced by computationally efficient spectrally equivalent \emph{approximations}, which then yields a computationally efficient optimal preconditioner for $\hat \bA$. 

We first consider the block $\hat \bA_0$ that corresponds to the Galerkin projection onto the global $H_0^1(\Omega)$-conforming finite element space $V_h$. We have
\begin{equation} \label{eqq3}
  \bx_0^T \hat \bA_0 \bx_0 =\|P_0\hat u_h\|_a^2 \sim \|P_0\hat u_h\|_b^2=\|(u_0,0)\|_b^2 = \sum_{i=1}^2 \|\nabla u_0\|_{\Omihex}^2 \sim \|\nabla u_0\|_{\Omega}^2.
\end{equation}
 It is natural to consider a spectrally equivalent preconditioner, denoted by $\bB_0$, for the interface problem \eqref{LRPAPER:eq:ellweakform} discretized in the standard conforming finite element space $V_h$, i.e., $\bB_0$ satisfies $\bx_0^T\bB_0 \bx_0 \sim (\alpha \nabla u_0,\nabla u_0)_\Omega$, with $u_0 = \sum_{j \in I_0} x_{0,j} \phi_j$. An option for such a $\bB_0$ is a multigrid preconditioner.  From \eqref{eqq3} it follows that $\bB_0$ is then also uniformly spectrally equivalent to $ \hat \bA_0$, i.e., $\bB_0 \sim \hat \bA_0$.
 
 We finally consider computationally efficient optimal preconditioners for the block $\hat \bA_1$, which corresponds to the local correction space $V^\Gamma$. 
 \begin{lemma}\label{LemD} For  $\bD_1:={\rm diag}(\hat \bA_1)$ the uniform spectral equivalence
 \[
   \bD_1 \sim \hat \bA_1,
 \]
holds. 
 \end{lemma}
 \begin{proof}
 From Lemma \ref{lemA} it follows that $\hat \bA_1$ is spectrally equivalent to a mass matrix and it is well-known that the diagonally scaled mass matrix has a uniformly bounded spectral condition number. For completeness we give the details. 
 Recall the relation between $\bx_1=(x_{1,j})_{j\in I^\Gamma}$ and $u^\Gamma$ that is given by $u^\Gamma = \sum_{j \in I^\Gamma} x_{1,j} \phi_{j}^\Gamma$. Using Lemma~\ref{lemA} we get
 \[
 \begin{split}
  \bx_1^T \hat \bA_1 \bx_1  =   \|(0,u^\Gamma)\|_a^2   \sim \|h^{-1}u^\Gamma\|_{\OmhGamma}^2.
 \end{split}
 \]
 For  $T \in \cTG$ we denote by $N(T)\subset I^\Gamma$ the subset of indices with corresponding nodes in $T$. Standard arguments yield that $\|u^\Gamma\|_T^2 \sim |T| \sum_{j \in N(T)} x_{1,j}^2$ holds. Using this  we get
  \begin{equation} \label{eq3} \begin{split}
   \bx_1^T \hat \bA_1 \bx_1  \sim  \sum_{T \in \cTG} h_T^{-2} |T|  \sum_{j \in N(T)} x_{1,j}^2. \end{split}
  \end{equation}
For $j \in I^\Gamma$ define $\hat u_j:=(0,\phi_{j}^\Gamma)$. Hence, $(D_1)_{j,j}=(\hat A_1)_{j,j}= \|P_1 \hat u_j\|_a^2$ and
\[
 \bx_1^T \bD_1 \bx_1 = \sum_{j \in I^\Gamma} (D_1)_{j,j} x_{1,j}^2 =\sum_{j \in I^\Gamma}\|P_1 \hat u_j\|_a^2 x_{1,j}^2 \sim 
 \sum_{j \in I^\Gamma}\|(0,\phi_{j}^\Gamma)\|_b^2 x_{1,j}^2.
\]
Using Lemma~\ref{lemA} we get  $\|(0,\phi_{j}^\Gamma)\|_b^2
\sim \|h^{-1}\phi_{j}^\Gamma\|_{\OmhGamma}^2\sim \sum_{T \subset {\rm  supp}(\phi_{j}^\Gamma)} h_T^{-2} |T|$
and thus we get
\begin{equation} \label{eq4} \begin{split}
 \bx_1^T \bD_1 \bx_1 \sim  \sum_{j \in I^\Gamma}x_{1,j}^2\sum_{T \subset {\rm  supp}(\phi_{j}^\Gamma)} h_T^{-2} |T|
   \sim \sum_{T \in \cTG}h_T^{-2} |T|\sum_{j \in N(T)} x_{1,j}^2 .
\end{split} \end{equation}
Comparing \eqref{eq3} and \eqref{eq4} we obtain the spectral equivalence.
   \end{proof}
\ \\
\begin{corollary}
 With $\bD_1={\rm diag}(\hat \bA_1)$, the matrix $\bD_1^{-\frac12} \hat \bA_1 \bD_1^{-\frac12}$ has a uniformly bounded spectral condition number. The scaling with $\bD_1^{-\frac12}$ can be deleted if the triangulations $\{\cTG\}_{h >0}$ are quasi-uniform.
\end{corollary}

Thus the solves $\hat \bA_1 \bx_1 = \hat{\mathbf{b}}_1$ in the evaluation of the exact preconditioner $\bD_A$ can be replaced by inexact solves of the scaled system $\bD_1^{-\frac12}\hat \bA_1\bD_1^{-\frac12} \tilde \bx_1 =\bD_1^{-\frac12} \hat{\mathbf{b}}_1$, $\tilde \bx_1=\bD_1^{\frac12} \bx_1$, using only \emph{a few iterations of a basic iterative method}, for example, of a symmetric Gauss-Seidel method. Note that the dimension of the matrix $ \hat \bA_1$ is much smaller than the dimension of $\hat \bA_0$. Hence, for  optimal efficiency of the preconditioner for $\hat \bA$ one should solve the (scaled) block system $\bD_1^{-\frac12}\hat \bA_1\bD_1^{-\frac12} \tilde \bx_1 =\bD_1^{-\frac12} \hat{\mathbf{b}}_1$, ``sufficiently accurate'', in order to avoid that a too poor preconditioning of the $\hat \bA_1$-block becomes the bottleneck.  

\begin{remark}\label{RemFDpc} \em
Based on the stable splitting result \eqref{eq:mainaFD} the same approach can be applied to derive optimal block Jacobi preconditioners for the fictitious domain discretization. In that case the $\hat \bA_0$ ``global'' block corresponds to a finite element discretization of the Laplace problem in  $V_1^-:=\mathrm{span}\{\phi_j\}_{j\in I_1\setminus I_1^\Gamma}$ with \emph{homogeneous Dirichlet boundary condition} on the boundary of the domain formed by these basis functions. As spectrally equivalent preconditioner $\bB_0$ for this block one can again use a multigrid solver. The other diagonal block $\hat \bA_1$ corresponds to Galerkin discretization in $V_1^\Gamma$ and the result in Lemma~\ref{LemD} implies that the diagonally scaled version of this matrix has a uniformly bounded condition number.
\end{remark}

\section{Numerical experiments} \label{sectNumExp}

In this section we present results of numerical experiments for the Poisson interface problem in 2D and 3D  and for the Poisson fictitious domain problem in 3D. All 2D numerical experiments\footnote{The 2D code is available via DOI 10.5281/zenodo.7249209, cf. \cite{zenodo2d}.} in Section~\ref{ss:PoissonIF2d} have been performed with NGSolve using the ngsxfem addon \cite{ngsxfem,ngsolve}.
All 3D numerical experiments\footnote{The 3D code is available via DOI 10.5281/zenodo.7257807, cf. \cite{zenodo3d}.} in Sections~\ref{ss:PoissonIF3d} and \ref{ss:PoissonFD} have been performed with the DROPS package \cite{DROPS}.

The analysis above leads to the following preconditioners for the linear system in \eqref{Axisb} (and its fictious domain analogon).
Preconditioners of $\hat \bA_i$  are denoted by  $\MAT{B}_i$, $i=0,1$.
  We define the block Jacobi preconditioners 
  \begin{align} \label{preco}
    \MAT{P}_{\MAT{A}} & := \bD_{A} = \begin{pmatrix} \hat\bA_0 & \MAT{0}  \\ \MAT{0}  & \hat \bA_1 \end{pmatrix}, \quad
    \MAT{P}_{\MAT{D}}  := \begin{pmatrix} \hat\bA_0 & \MAT{0}  \\ \MAT{0}  &  \bB_1  \end{pmatrix}, \quad
     \MAT{P}_{\MAT{B}} := \begin{pmatrix} \bB_0 & \MAT{0}  \\ \MAT{0}  &  \bB_1  \end{pmatrix}.
 \end{align}
 For $\MAT{B}_0^{-1}$ we use a multigrid solver  applied to $\hat\bA_0$. A more precise specification of this solver is given in the subsections below. For $\MAT{B}_1^{-1}$ we use the symmetric Gauss-Seidel preconditioner (one iteration) applied to $\hat \bA_1$.  
 In the following, we apply a preconditioned conjugate gradient (PCG) method  to the linear system \eqref{Axisb} and examine different choices of preconditioners $\MAT{P}$. Starting with $\bx^0=0$, the PCG iteration is stopped when the preconditioned residual is reduced by a factor $\mathrm{tol}=10^{-6}$, i.e. 
\begin{align}
  \|\MAT{P}^{-1}(\hat{\MAT{A}}\bx^k - \hat{\textbf{b}})\|_2 &\leq \mathrm{tol}\, \|\MAT{P}^{-1}(\hat{\MAT{A}}\bx^0 - \hat{\textbf{b}})\|_2, \label{eq:stop-resid}
\end{align}
with $\|\cdot\|_2$ the Euclidean norm. In Section~\ref{ss:PoissonIF2d} we also consider another stopping criterion, namely the one in \eqref{eq:stop-err}. The reason why we use this alternative  is explained in that section.

\subsection{Poisson interface problem, 2D} \label{ss:PoissonIF2d}

For the subdomain $\Omega_1$ we take the unit circle w.r.t. $\|\cdot\|_4$, $\Omega_1:=\{x\in\R^2:~\|x-x_0\|_4 \leq 1\}$ around midpoint $x_0\in\R^2$ and the domain $\Omega:=[-1.5,1.5]^2 \supset \Omega_1$. For $x\in\R^3$ we define $\hat x:= x-x_0$. If not stated differently, we use $x_0=(0.001, 0.002, 0,003)^T$ in the remainder to avoid symmetry effects. We choose an $\alpha$-dependent function $u:\Omega\to\R$, $u(x)|_{\Omega_i}:= \alpha_i^{-1}(3\hat x_1^2 \hat x_2 - \hat x_2^3) (\exp(1-\|\hat x\|_4^2) -1)$, $i=1,2$, with $\alpha_1=1$, $\alpha_2=10$. The right-hand side $f$ and boundary data $g=u$ are chosen such that $u$ is the solution of \eqref{LRPAPER:eq:ellmodel} on $\Omega$. 
For the construction of a family of triangulations, an initial triangulation $\T_0$ of $\Omega$ with mesh size $h_0=3/8$ is constructed. Applying successive uniform refinement yields the grids $\T_\ell$ with refinement levels $\ell=1,\ldots,6$ and corresponding grid sizes $h_\ell=2^{-\ell}\cdot\frac{3}{8}$. 

We use linear and quadratic finite elements ($k=1,2$). For $k=2$, in order to obtain a sufficiently accurate interface approximation, we apply a suitable isoparametric mapping to the triangles intersected by the interface; cf.~\cite{LehrenfeldHO} for more details. Corresponding finite element spaces $V_{h_\ell}$ are constructed on the respective grids $\T_\ell$, $\ell=0,1,\ldots,6$. Tables~\ref{tab:gridIF2d} and \ref{tab:gridIF2d-2} report for the different levels the dimensions of the global and local space, cf. \eqref{locglob}, $N_0=\dim V_h$  and $N_1=\dim V_h^\Gamma$, respectively. We observe that $N_0$ and $N_1$ grow with the expected  factors of approximately $4$ and $2$, respectively.

\begin{table}[ht!]
\begin{minipage}{0.43\textwidth}\centering
\begin{tabular}{crr}
\toprule
$\ell$ & $N_0$ & $N_1$ \\
\midrule
 0 & 54     & 42 \\
 1 & 245    & 83 \\
 2 & 1,041  & 161  \\
 3 & 4,289  & 325\\
 4 & 17,409 & 657 \\
 5 & 70,145 & 1,311\\
 6 & 281,601& 2,627\\
\bottomrule
\end{tabular}
\caption{Dimensions $N_0, N_1$ for different refinement levels $\ell$ for the 2D Poisson interface problem, $k=1$.}
\label{tab:gridIF2d}
\end{minipage} \hfill
\begin{minipage}{0.54\textwidth}\centering
\begin{tabular}{cllll}
\toprule
$\ell$ & $\|u-u_h\|_0$ & order & $\|u-u_h\|_1$ & order \\
\midrule
 0 & 2.65E-01 &      & 2.05E+00 &     \\
 1 & 1.12E-01 & 1.24 & 1.36E+00 & 0.59 \\
 2 & 3.09E-02 & 1.86 & 6.96E-01 & 0.96 \\
 3 & 6.98E-03 & 2.14 & 3.43E-01 & 1.02 \\
 4 & 1.67E-03 & 2.06 & 1.70E-01 & 1.01 \\
 5 & 4.06E-04 & 2.04 & 8.49E-02 & 1.00 \\
 6 & 9.96E-05 & 2.03 & 4.24E-02 & 1.00 \\
\bottomrule
\end{tabular}
\caption{Discretization errors w.r.t. $L^2$ and $H^1$ norm for different refinement levels $\ell$ for the 2D Poisson interface problem, $k=1$.}
\label{tab:convIF2d}
\end{minipage}
\end{table}

For the ghost-penalty term we use a facet-based variant which is advantageous in the higher order case, cf. Remark~6 in \cite{Preuss2018}. Choosing the Nitsche parameter $\gamma=10$ and ghost penalty parameter $\beta=0.1$, we obtain numerical solutions $u_{h_\ell}\in V_{h_\ell}$ of the discrete problem \eqref{LRPAPER:Nitsche1}, with  discretization errors w.r.t. the $L^2$ and $H^1$ norm as  in Tables~\ref{tab:convIF2d} and \ref{tab:convIF2d-2}. We clearly observe optimal convergence rates in the $L^2$  and in the $H^1$ norm. 

\begin{table}[ht!]
\begin{minipage}{0.43\textwidth}\centering
\begin{tabular}{crr}
\toprule
$\ell$ & $N_0$ & $N_1$ \\
\midrule
 0 & 245        & 126 \\
 1 & 1,041      & 249 \\
 2 & 4,289      & 483  \\
 3 & 17,409     & 975\\
 4 & 70,145     & 1,971 \\
 5 & 281,601    & 3,933\\
 6 & 1,128,449  & 7,881\\
\bottomrule
\end{tabular}
\caption{Dimensions $N_0, N_1$ for different refinement levels $\ell$ for the 2D Poisson interface problem, $k=2$.}
\label{tab:gridIF2d-2}
\end{minipage} \hfill
\begin{minipage}{0.54\textwidth}\centering
\begin{tabular}{cllll}
\toprule
$\ell$ & $\|u-u_h\|_0$ & order & $\|u-u_h\|_1$ & order \\
\midrule
 0 & 2.50E-02 &      & 5.32E-01&     \\
 1 & 5.51E-03 & 2.18 & 2.02E-01& 1.39 \\
 2 & 7.52E-04 & 2.87 & 5.00E-02& 2.02 \\
 3 & 1.00E-04 & 2.91 & 1.26E-02& 1.98 \\
 4 & 1.30E-05 & 2.95 & 3.18E-03& 1.99 \\
 5 & 1.64E-06 & 2.98 & 7.96E-04& 2.00 \\
 6 & 2.07E-07 & 2.99 & 1.99E-04& 2.00 \\
\bottomrule
\end{tabular}
\caption{Discretization errors w.r.t. $L^2$ and $H^1$ norm for different refinement levels $\ell$ for the 2D Poisson interface problem, $k=2$.}
\label{tab:convIF2d-2}
\end{minipage}
\end{table}

We present results for the symmetric Gauss-Seidel preconditioner $\MAT{P}_\text{SGS}$ and the block Jacobi preconditioners $\MAT{P}_\MAT{A}, \MAT{P}_\MAT{D}, \MAT{P}_\MAT{B}$ defined in \eqref{preco}. For $\MAT{B}_0^{-1}$ we choose one multigrid cycle with forward/backward Gauss-Seidel smoothing applied to $\hat\bA_0$.
The condition numbers $\kappa_2(\hat{\MAT{A}})=\|\hat{\MAT{A}}\|_2\|\hat{\MAT{A}}^{-1}\|_2$ and PCG iteration numbers for different refinement levels $\ell$ are reported in Tables~\ref{tab:pc-gridrefIF2d} and \ref{tab:pc-gridrefIF2d-2}.
\begin{table}[ht!]  \centering
\begin{tabular}{crrrrr}
\toprule
$\ell$ & $\kappa_2(\hat\bA)$ &  \multicolumn{4}{c}{PCG iterations}\\
&& $\MAT{P}_\text{SGS}$ & $\MAT{P}_\MAT{A}$ & $\MAT{P}_\MAT{D}$ & $\MAT{P}_\MAT{B}$ \\
\midrule
 0 & 4.50E+03 & 21  & 20  & 27 & 27  \\
 1 & 2.24E+04 & 25  & 23  & 29 & 29\\
 2 & 9.75E+03 & 42  & 21  & 27 & 27\\
 3 & 1.22E+04 & 79  & 22  & 27 & 28\\
 4 & 5.33E+04 & 151 & 22  & 26 & 27\\
 5 & 2.28E+05 & 291 & 21  & 26 & 27\\
 6 & 9.01E+05 & 535 & 20  & 24 & 26 \\
\bottomrule
\end{tabular}
\caption{Condition numbers and PCG iteration numbers for different preconditioners and varying grid refinement levels $\ell$ for the 2D Poisson interface problem, $k=1$.}
\label{tab:pc-gridrefIF2d}
\end{table}

We first discuss the  case $k=1$. For finer grid levels $\ell\geq 3$ the condition number $\kappa_2(\hat\bA)$ behaves like $\sim h^{-2}$ as for  stiffness matrices of standard conforming finite element discretizations of a Poisson problem. 
For the symmetric Gauss-Seidel preconditioner $\MAT{P}_\text{SGS}$, on the finer grid levels the iteration numbers grow approximately like $h^{-1}$.
For the block preconditioners $\MAT{P}_\MAT{A}, \MAT{P}_\MAT{D}, \MAT{P}_\MAT{B}$, we observe almost constant iteration numbers for increasing level $\ell$. For each grid level, the iteration numbers of the block preconditioners $\MAT{P}_\MAT{D}$ and $\MAT{P}_\MAT{B}$ are very similar. Note the very small increase in iteration numbers when we change from the exact block preconditioner $\MAT{P}_\MAT{A}$ to the inexact ones $\MAT{P}_\MAT{D}$ and $\MAT{P}_\MAT{B}$. The third preconditioner, $\MAT{P}_\MAT{B}$, is the only one with computational costs $\mathcal{O}(N)$, $N:=N_0+N_1$, with a constant independent of $\ell$.

\begin{table}[ht!]  \centering
\begin{tabular}{crrrrr}
\toprule
$\ell$ & $\kappa_2(\hat\bA)$ &  \multicolumn{4}{c}{PCG iterations}\\
&& $\MAT{P}_\text{SGS}$ & $\MAT{P}_\MAT{A}$ & $\MAT{P}_\MAT{D}$ & $\MAT{P}_\MAT{B}$ \\
\midrule
 0 & 1.59E+05 & 99  & 51   & 112 & 113  \\
 1 & 3.94E+06 & 113 & 95   & 150 & 154\\
 2 & 2.13E+06 & 105 & 101  & 138 & 140\\
 3 & 1.49E+06 & 105 & 123  & 132 & 137\\
 4 & 1.80E+06 & 151 & 110  & 121 & 126\\
 5 & 1.77E+06 & 290 & 97   & 107 & 111\\
 6 & 1.63E+06 & 533 & 84   & 101 & 106 \\
\bottomrule
\end{tabular}
\caption{Condition numbers and PCG iteration numbers for different preconditioners and varying grid refinement levels $\ell$ for the 2D Poisson interface problem, $k=2$.}
\label{tab:pc-gridrefIF2d-2}
\end{table}

For $k=2$ the condition numbers are larger than for the linear case, but do not show a scaling with $h^{-2}$. For $\ell\geq 4$ the iteration numbers of the symmetric Gauss-Seidel preconditioner $\MAT{P}_\text{SGS}$ again grow like $h^{-1}$. Compared to $k=1$, the iteration numbers of the block preconditioners show larger variations w.r.t. the grid refinement level $\ell$.  On the finer levels $\ell\geq 4$, where the  iteration number of the symmetric Gauss-Seidel preconditioner has the expected $\ell$-dependent strong increase, we observe (as for the case $k=1$) a \emph{de}crease of the iteration number for the block preconditioners $\MAT{P}_\MAT{A}$, $\MAT{P}_\MAT{D}$, $\MAT{P}_\MAT{B}$. As before, on not too coarse levels the iteration numbers of $\MAT{P}_\MAT{D}$ and $\MAT{P}_\MAT{B}$ are very similar on each level $\ell$.

We now fix the grid refinement level $\ell=2$ and vary the midpoint $x_0= (\delta, 2\delta, 3\delta)$ of the ball $\Omega_1$ with $\delta\in[0,0.1]$, leading to different relative positions of $\Gamma$ within the background mesh $\T_2$. Note that $h_2=3/32 < 0.1$, so the interface is moved in $x$-direction (slightly) more than the width of one grid cell. The condition numbers with and without ghost penalty stabilization as a function of $\delta$ are shown in Figure~\ref{fig:cond-relposIF2d-2} for $k=2$. While the condition numbers for $\beta=0$ oscillate on a high level (due to ``bad cut'' situations encountered for the specific $\delta$), the condition numbers for $\beta>0$ are much smaller around $10^6$ and have much less fluctuations. We repeated the experiments for $k=1$ and $\ell=4$. The condition numbers and PCG iteration numbers for different choices of $\delta$ are reported in Table~\ref{tab:pc-relposIF2d-1}. 
We observe that for varying $\delta$,  due to the ghost penalty stabilization, the condition number $\kappa_2(\hat\bA)$ has the same order of magnitude. The PCG iteration numbers for all considered preconditioners  show only small fluctuations for varying $\delta$.
\begin{figure}
\centering
\includegraphics[width=0.7\textwidth]{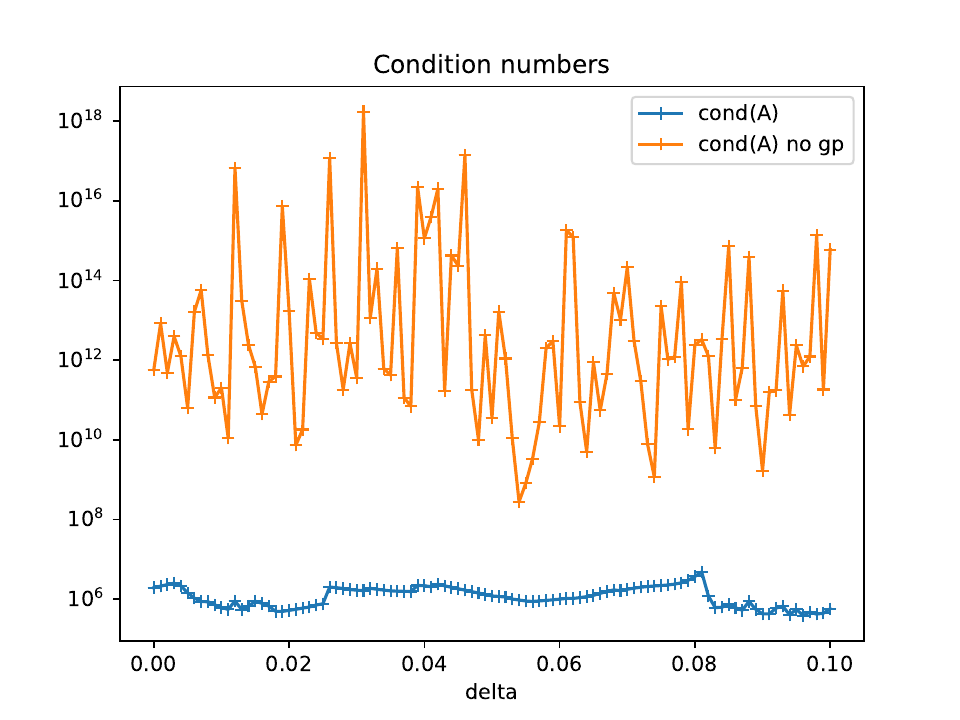}
\caption{Condition numbers with and without ghost penalty stabilization (blue and orange crosses, resp.) for varying midpoint $x_0= (\delta, 2\delta, 3\delta)$ for the 2D Poisson interface problem on grid refinement level $\ell=2$, $k=2$.}
\label{fig:cond-relposIF2d-2}
\end{figure} 
\begin{table}[ht!]\centering
\begin{tabular}{lllllll}
\toprule
$\delta$ & $\kappa_2(\hat\bA)$ &  \multicolumn{4}{c}{PCG iterations}\\
&& $\MAT{P}_\text{SGS}$ & $\MAT{P}_\MAT{A}$ & $\MAT{P}_\MAT{D}$ & $\MAT{P}_\MAT{B}$ \\
\midrule
 0    & 5.39E+04 & 151 & 22 & 26 & 27\\
 0.02 & 5.57E+04 & 151 & 22 & 27 & 28\\
 0.04 & 5.73E+04 & 152 & 22 & 27 & 28\\
 0.06 & 5.26E+04 & 154 & 21 & 26 & 27\\
 0.08 & 5.14E+04 & 156 & 19 & 27 & 28\\
 0.10 & 5.54E+04 & 158 & 21 & 25 & 26\\
\bottomrule
\end{tabular}
\caption{Condition numbers and PCG iteration numbers for different preconditioners and varying midpoint $x_0= (\delta, 2\delta, 3\delta)$ for the 2D Poisson interface problem on grid refinement level $\ell=4$, $k=1$.}
\label{tab:pc-relposIF2d-1}
\end{table}

For our analysis to be applicable it is essential that we consider the Nitsche method \emph{with stabilization}, i.e., $\beta >0$ in \eqref{defah}. As noted above,  cf. Figure~\ref{fig:cond-relposIF2d-2}, for $\beta=0$ the condition numbers $\kappa_2(\hat\bA)$ can be extremely large.  However, results of numerical experiments (not presented here) indicate that for $k=1$ this does not significantly affect the PCG iteration numbers, which show a similar behavior as for the case with $\beta>0$. These results are consistent with the ones presented in \cite{lehrenfeld2017optimal}. 

For higher order finite elements ($k\geq 2$) the situation is significantly different. First note that for $\beta=0$ the condition numbers $\kappa_2(\hat\bA)$ can be extremely large and we cannot show spectral equivalence of $\hat\bA$ and any of the block preconditioners $\MAT{P}_\MAT{A}, \MAT{P}_\MAT{D}, \MAT{P}_\MAT{B}$. Hence, there is no reason why the preconditioned residual should be a good error measure and thus the criterion  \eqref{eq:stop-resid} is not satisfactory. Therefore,  for the case $k=2$ and $\beta=0$, we  consider a different stopping criterion 
\begin{align}
  \| \bx^k - \bx^* \|_2 &\leq \mathrm{tol}\, \| \bx^0 - \bx^* \|_2,\label{eq:stop-err}
\end{align}
with $\bx^* = \hat\bA^{-1}\hat{\textbf{b}}$ and $\mathrm{tol} = 10^{-6}$. Note that this criterion requires the exact solution $\bx^*$ (which is not available in practical applications). This criterion is optimal in the sense that it yields control over the error norm $\|u_h^k - u_h\|_{L^2(\Omega)} = \langle \hat{\mathbf{M}}(\bx^k-\bx^*), \bx^k-\bx^*\rangle^{1/2} \sim \|\bx^k-\bx^*\|_2$, with  the mass matrix $\hat{\mathbf{M}}$ on $V_h\times V_h^\Gamma$,  which has a uniform condition number bound $\kappa_2(\hat{\mathbf{M}})\sim1$. Here $u_h^k, u_h\in V_h\times V_h^\Gamma$ denote the finite element functions corresponding to the coefficient vectors $\bx^k, \bx^*$, respectively.

In Table~\ref{tab:pc-noGP-2} condition numbers are reported for $k=2$ and $\beta=0$, and PCG iterations for the stopping criterion \eqref{eq:stop-err} are shown for $\beta=0$ and $\beta=0.1$. Comparing with Table~\ref{tab:pc-gridrefIF2d-2}, we see that the condition numbers are dramatically increased for $\beta=0$. We notice that for $\beta=0$ the PCG solver does not converge within $10,000$ iterations (marked by DIV) for grid refinement levels $\ell\geq 3$, while the iteration numbers for $\beta=0.1$ show a similar behavior as the ones in Table~\ref{tab:pc-gridrefIF2d-2}, which are based on the stopping criterion \eqref{eq:stop-resid}.
The smaller numbers in Table~\ref{tab:pc-gridrefIF2d-2} compared to Table~\ref{tab:pc-noGP-2} can be explained by the different stopping criteria. As an example, the computed approximation $\bx^{107}$ with preconditioner $\MAT{P}_\MAT{D}$ on level $\ell=5$ (i.e., 107 iterations in Table~\ref{tab:pc-gridrefIF2d-2}) has a relative error $\frac{\|\bx^{107}- \bx^\ast\|_2}{\|\bx^0-\bx^\ast\|_2} = 1.65\cdot 10^{-4}$.
We conclude that for higher order elements the ghost penalty stabilization is essential for the efficiency of the preconditioned iterative solver studied in this paper, not only in theory but also in practice.

\begin{table}[ht!]  \centering
\begin{tabular}{crrrrr}
\toprule
$\ell$ & $\kappa_2(\hat\bA)$ &  \multicolumn{4}{c}{PCG iterations}\\
&& $\MAT{P}_\text{SGS}$ & $\MAT{P}_\MAT{A}$ & $\MAT{P}_\MAT{D}$ & $\MAT{P}_\MAT{B}$ \\
\midrule
 0 & 1.52E+11 & 281 & 82   & 328 & 327 \\
 1 & 1.13E+14 & DIV & 9,652 & DIV & DIV \\
 2 & 8.31E+12 & 1,025& 527  & 1,393&1,395 \\
 3 & 4.40E+16 & DIV & DIV  & DIV & DIV \\
 4 & 2.87E+15 & DIV & DIV  & DIV & DIV \\
 5 & 1.72E+17 & DIV & DIV  & DIV & DIV \\
 6 & 2.31E+18 & DIV & DIV  & DIV & DIV \\
\bottomrule
\end{tabular}
\qquad
\begin{tabular}{crrrr}
\toprule
$\ell$ &  \multicolumn{4}{c}{PCG iterations}\\
& $\MAT{P}_\text{SGS}$ & $\MAT{P}_\MAT{A}$ & $\MAT{P}_\MAT{D}$ & $\MAT{P}_\MAT{B}$ \\
\midrule
 0 & 137 &  66  & 163 & 164  \\
 1 & 189 & 149  & 235 & 240\\
 2 & 187 & 175  & 243 & 247\\
 3 & 198 & 239  & 284 & 289\\
 4 & 226 & 254  & 280 & 287\\
 5 & 444 & 264  & 296 & 300\\
 6 & 924 & 284  & 308 & 314 \\
\bottomrule
\end{tabular}
\caption{Impact of missing ghost penalty stabilization on condition numbers and PCG iteration numbers w.r.t.~\eqref{eq:stop-err} for different preconditioners and varying grid refinement levels $\ell$ for the 2D Poisson interface problem, $k=2$. Left: $\beta=0$ (no ghost penalty), right: $\beta=0.1$.}
\label{tab:pc-noGP-2}
\end{table}

\subsection{Poisson interface problem, 3D} \label{ss:PoissonIF3d}

For the subdomain $\Omega_1$ we choose the unit ball $\Omega_1:=B_1(x_0)=\{x\in\R^3:~\|x-x_0\|_2 \leq 1\}$ around midpoint $x_0\in\R^3$ and the domain $\Omega:=[-1.5,1.5]^3 \supset \Omega_1$. For $x\in\R^3$ we define $\hat x:= x-x_0$ with $x_0=(0.001, 0.002, 0,003)^T$ to avoid symmetry effects. We choose an $\alpha$-dependent function $u:\Omega\to\R$, $u(x)|_{\Omega_i}:= \alpha_i^{-1}(3\hat x_1^2 \hat x_2 - \hat x_2^3) (\exp(1-\|\hat x\|_2^2) -1)$, $i=1,2$, with $\alpha_1=1$, $\alpha_2=10$. The right-hand side $f$ and boundary data $g=u$ are chosen such that $u$ is the solution of \eqref{LRPAPER:eq:ellmodel} on $\Omega$. 
For the construction of a family of tetrahedral triangulations,  the domain $\Omega$ is partitioned into $4\times 4\times 4$ cubes, where each cube is further subdivided into 6 tetrahedra, forming an initial tetrahedral triangulation $\T_0$ of $\Omega$. Applying successive uniform refinement yields the grids $\T_\ell$ with refinement levels $\ell=1,\ldots,6$ and corresponding grid sizes $h_\ell=2^{-\ell}\cdot\frac{3}{4}$. 

We use linear finite elements ($k=1$) and construct finite element spaces $V_{h_\ell}$ on the respective grids $\T_\ell$, $\ell=0,1,\ldots,6$. Table~\ref{tab:gridIF} reports for the different levels the dimensions of the global and local space, cf. \eqref{locglob}, $N_0=\dim V_h$  and $N_1=\dim V_h^\Gamma$, respectively. We observe that $N_0$ and $N_1$ grow with the expected  factors of approximately $8$ and $4$, respectively.

\begin{table}[t]
\begin{minipage}{0.43\textwidth}\centering
\begin{tabular}{crr}
\toprule
$\ell$ & $N_0$ & $N_1$ \\
\midrule
 0 & 27 &	 27 \\
 1 & 343 &	 208  \\
 2 & 3,375 &	 844  \\
 3 & 29,791 &	3,373 \\
 4 & 250,047 & 13,580	 \\
 5 & 2,048,383 & 54,191\\
 6 & 16,581,375 &216,548  \\
\bottomrule
\end{tabular}
\caption{Dimensions $N_0, N_1$ for different refinement levels $\ell$ for the 3D Poisson interface problem, $k=1$.}
\label{tab:gridIF}
\end{minipage} \hfill
\begin{minipage}{0.54\textwidth}\centering
\begin{tabular}{cllll}
\toprule
$\ell$ & $\|u-u_h\|_0$ & order & $\|u-u_h\|_1$ & order \\
\midrule
 0 & 2.40E-01 &     & 1.76E+00&     \\
 1 & 1.19E-01 & 1.01& 1.13E+00& 0.64 \\
 2 & 4.43E-02 & 1.43& 6.41E-01& 0.82 \\
 3 & 1.19E-02 & 1.90& 3.30E-01& 0.96 \\
 4 & 2.91E-03 & 2.03& 1.67E-01& 0.98 \\
 5 & 7.04E-04 & 2.05& 8.42E-02& 0.99 \\
 6 & 1.72E-04 & 2.03& 4.22E-02& 1.00 \\
\bottomrule
\end{tabular}
\caption{Discretization errors w.r.t. $L^2$ and $H^1$ norm for different refinement levels $\ell$ for the 3D Poisson interface problem, $k=1$.}
\label{tab:convIF}
\end{minipage}
\end{table}

Choosing the Nitsche parameter $\gamma=10$ and ghost penalty parameter $\beta=0.1$, we obtain numerical solutions $u_{h_\ell}\in V_{h_\ell}$ of the discrete problem \eqref{LRPAPER:Nitsche1}, with  discretization errors w.r.t. the $L^2$ and $H^1$ norm as  in Table~\ref{tab:convIF}. We clearly observe optimal convergence rates in the $L^2$  and in the $H^1$ norm. 

We present results for the symmetric Gauss-Seidel preconditioner $\MAT{P}_\text{SGS}$ and the block Jacobi preconditioners $\MAT{P}_\MAT{A}, \MAT{P}_\MAT{D}, \MAT{P}_\MAT{B}$ defined in \eqref{preco}. We choose for $\MAT{B}_0^{-1}$ 3~multigrid sweeps (V-cycle) with symmetric Gauss-Seidel smoothing applied to $\hat\bA_0$.
We use the stopping criterion \eqref{eq:stop-resid}. The performance of the different preconditioners is very similar to that in the 2D case reported above, cf. Table~\ref{tab:pc-gridrefIF2d}. The condition numbers $\kappa_2(\hat{\MAT{A}})=\|\hat{\MAT{A}}\|_2\|\hat{\MAT{A}}^{-1}\|_2$ and PCG iteration numbers for different refinement levels $\ell$ are given in Table~\ref{tab:pc-gridrefIF}.
\begin{table}[ht!]  \centering
\begin{tabular}{crrrrr}
\toprule
$\ell$ & $\kappa_2(\hat\bA)$ &  \multicolumn{4}{c}{PCG iterations}\\
&& $\MAT{P}_\text{SGS}$ & $\MAT{P}_\MAT{A}$ & $\MAT{P}_\MAT{D}$ & $\MAT{P}_\MAT{B}$ \\
\midrule
 0 & 8.77E+01 & 12  & 14  & 17 & 17  \\
 1 & 9.79E+02 & 18  & 22  & 24 & 24\\
 2 & 1.28E+03 & 21  & 23  & 26 & 26\\
 3 & 2.33E+03 & 34  & 25  & 26 & 26\\
 4 & 9.13E+03 & 63  & 23  & 26 & 26\\
 5 & 3.69E+04 & 109 & 22  & 25 & 25\\
 6 & 1.50E+05 & 207 & 21  & 23 & 23 \\
\bottomrule
\end{tabular}
\caption{Condition numbers and PCG iteration numbers for different preconditioners and varying grid refinement levels $\ell$ for the 3D Poisson interface problem, $k=1$.}
\label{tab:pc-gridrefIF}
\end{table}

For finer grid levels $\ell\geq 4$ the condition number $\kappa_2(\hat\bA)$ behaves like $\sim h^{-2}$ as for  stiffness matrices of standard conforming finite element discretizations of a Poisson problem. 
For the symmetric Gauss-Seidel preconditioner $\MAT{P}_\text{SGS}$, on the finer grid levels the iteration numbers grow approximately like $h^{-1}$.
For the block preconditioners $\MAT{P}_\MAT{A}, \MAT{P}_\MAT{D}, \MAT{P}_\MAT{B}$, we observe almost constant iteration numbers for increasing level $\ell$. For each grid level, the iteration numbers of the block preconditioners are very similar (and even the same for $\MAT{P}_\MAT{D}$ and $\MAT{P}_\MAT{B}$). Note the very small increase in iteration numbers when we change from the exact block preconditioner $\MAT{P}_\MAT{A}$ to the inexact ones $\MAT{P}_\MAT{D}$ and $\MAT{P}_\MAT{B}$.  

In the analysis presented in this paper we derived uniform spectral equivalence under the assumption that the family of simplicial triangulations is shape regular, but not necessarily quasi-uniform. This motivates  the next experiment, in which for the same 3D interface problem as above we apply, starting from the initial triangulation $\T_0$, a successive \emph{local} refinement of the tetrahedra intersected by $\Gamma$ to obtain a hierarchy of refined grids $\T_\ell$, $\ell=0,1,\ldots$. This refinement process leads to a family of tetrahedral triangulations that are shape regular, but not quasi-uniform. The dimensions of the linear finite element spaces $V_{h_\ell}$ on the respective grids $\T_\ell$, $\ell=0,1,\ldots,6$ are reported in Table~\ref{tab:pc-gridrefIF-adap}. We observe that both $N_0$ and $N_1$ grow with a factor of 4, as only elements in the vicinity of the 2D interface $\Gamma$ are refined. As before, for the symmetric Gauss-Seidel preconditioner the number of PCG iterations grows with increasing grid refinement level $\ell$, but not as fast as for the case of uniform refinement, cf. Table~\ref{tab:pc-gridrefIF}. For the block preconditioners we observe almost the same iteration numbers, regardless of local or uniform refinement.
\begin{table}[ht!]  \centering
\begin{tabular}{crrrrrrr}
\toprule
$\ell$ & $N_0$ & $N_1$ & $\kappa_2(\hat\bA)$ &  \multicolumn{4}{c}{PCG iterations}\\
&&&& $\MAT{P}_\text{SGS}$ & $\MAT{P}_\MAT{A}$ & $\MAT{P}_\MAT{D}$ & $\MAT{P}_\MAT{B}$ \\
\midrule
 0 &      27 &      27 & 8.77E+01 & 12 & 14  & 17 & 17  \\
 1 &     221 &     208 & 1.13E+03 & 18 & 23  & 25 & 25\\
 2 &   1,311 &     844 & 2.11E+03 & 19 & 25  & 26 & 26\\
 3 &   6,041 &   3,373 & 5.33E+03 & 27 & 25  & 27 & 27\\
 4 &  25,344 &  13,580 & 8.61E+03 & 40 & 24  & 26 & 26\\
 5 & 103,337 &  54,191 & 1.61E+04 & 57 & 23  & 24 & 24\\
 6 & 422,285 & 216,548 & 7.50E+04 & 75 & 21  & 23 & 23 \\
\bottomrule
\end{tabular}
\caption{Dimensions $N_0, N_1$, condition numbers and PCG iteration numbers for different preconditioners and varying grid refinement levels $\ell$ for the 3D Poisson interface problem on adaptively refined grids, $k=1$.}
\label{tab:pc-gridrefIF-adap}
\end{table}

\subsection{Poisson fictitious domain problem, 3D} \label{ss:PoissonFD}

We now consider the Poisson fictitious domain problem in \eqref{eq:discreteFD}. Let $\Omega$ and $\Omega_1$ be defined as in section~\ref{ss:PoissonIF3d}. For the function $u:\Omega\to\R$, $u(x):= (3\hat x_1^2 \hat x_2 - \hat x_2^3) \exp(1-\|\hat x\|_2^2)$, the right-hand side $f(x)=u(x)(-4\|\hat x\|_2^2 + 18)$ and boundary data $g=u$ are chosen such that $u$ is the solution of \eqref{eq:Poisson} on $\Omega_1$. 
For  discretization the same initial triangulation $\T_0$ of $\Omega$ as in section~\ref{ss:PoissonIF3d} is chosen. Applying an adaptive refinement algorithm, where all tetrahedra $T\in\T_0$ with $\mathrm{meas}_3(T\cap\Omega_1)>0$ 
are marked for regular refinement, we obtain the refined grid $\T_1$. Repeating this refinement process yields the grids $\T_\ell$ with refinement levels $\ell=2,\ldots,6$ and corresponding grid sizes $h_\ell=2^{-\ell}\cdot\frac{3}{4}$. 

We use linear finite elements ($k=1$) and construct finite element spaces $V_{h_\ell}$ on the respective grids $\T_\ell$, $\ell=0,1,\ldots,6$. Table~\ref{tab:gridFD} reports the numbers $N_0=\dim V_1^-$ (the number of grid points inside the fictitious domain) and $N_1=\dim V_1^\Gamma$ (the number of grid points on $\partial \Omega_{1,h}^\mathrm{ex}$) for different grid levels. We observe that $N_0$ and $N_1$ grow with the expected factors of approximately $8$ and $4$, respectively.

\begin{table}[ht!]
\begin{minipage}{0.43\textwidth}\centering
\begin{tabular}{crr}
\toprule
$\ell$ & $N_0$ & $N_1$ \\
\midrule
 0 & 7 &	44 \\
 1 & 81 &	140 \\
 2 & 619 &	500 \\
 3 & 5,070 &	1,844 \\
 4 & 40,642 &	7,102 \\
 5 & 325,444 &	27,714 \\
 6 & 2,602,948 &	109,510 \\
\bottomrule
\end{tabular}
\caption{Dimensions $N_0, N_1$ for different refinement levels $\ell$ for the 3D fictitious domain problem, $k=1$.}
\label{tab:gridFD}
\end{minipage} \hfill
\begin{minipage}{0.54\textwidth}\centering
\begin{tabular}{cllll}
\toprule
$\ell$ & $\|u-u_h\|_0$ & order & $\|u-u_h\|_1$ & order \\
\midrule
 0 & 2.19E-01&&1.26E+00&     \\
 1 & 5.95E-02&1.88&6.17E-01&1.04\\
 2 & 1.43E-02&2.05&3.12E-01&0.98 \\
 3 & 3.40E-03&2.08&1.56E-01&1.00 \\
 4 & 8.15E-04&2.06&7.81E-02&1.00 \\
 5 & 1.98E-04&2.04&3.91E-02&1.00 \\
 6 & 4.89E-05&2.02&1.96E-02&1.00 \\
\bottomrule
\end{tabular}
\caption{Discretization errors w.r.t. $L^2$ and $H^1$ norm for different refinement levels $\ell$ for the 3D fictitious domain problem, $k=1$.}
\label{tab:convFD}
\end{minipage}
\end{table}

Choosing $\gamma=10$ and $\beta=0.1$, we obtain numerical solutions $u_{h_\ell}\in V_{h_\ell}$ of the discrete problem \eqref{eq:discreteFD}, with  discretization errors w.r.t. the $L^2$ and $H^1$ norm as  in Table~\ref{tab:convFD}. Optimal convergence rates in the $L^2$  and in the $H^1$ norm are observed. 

We present results for the symmetric Gauss-Seidel preconditioner $\MAT{P}_\text{SGS}$ and the block Jacobi preconditioners defined in \eqref{preco}, where this time $\MAT{B}_0^{-1}$ denotes one iteration of an algebraic multigrid solver (HYPRE BoomerAMG \cite{HYPRE,HensonBoomerAMG}) applied to $\hat{\MAT{A}}_0$ and $\MAT{B}_1^{-1}$ denotes three symmetric Gauss-Seidel iterations.
The condition numbers $\kappa_2(\hat{\MAT{A}})$ and PCG iteration numbers (with stopping criterion \eqref{eq:stop-resid}) for different refinement levels $\ell$ are reported in Table~\ref{tab:pc-gridrefFD}.
\begin{table}[ht!]  \centering
\begin{tabular}{crrrrr}
\toprule
$\ell$ & $\kappa_2(\hat{\MAT{A}})$ &  \multicolumn{4}{c}{PCG iterations}\\
&& $\MAT{P}_\text{SGS}$ & $\MAT{P}_\MAT{A}$ & $\MAT{P}_\MAT{D}$ & $\MAT{P}_\MAT{B}$ \\
\midrule
 0 & 1.41E+02 &   8 &  9  & 9  & 9 \\
 1 & 1.03E+02 &   9 & 12  & 12 & 12 \\
 2 & 1.58E+02 &  13 & 11  & 12 & 12 \\
 3 & 2.97E+02 &  20 & 13  & 14 & 14 \\
 4 & 7.74E+02 &  34 & 13  & 14 & 14 \\
 5 & 3.11E+03 &  56 & 13  & 13 & 13\\
 6 & 1.26E+04 &  107 & 16 & 17 & 18 \\
\bottomrule
\end{tabular}
\caption{Condition numbers and PCG iteration numbers for different preconditioners and varying grid refinement levels $\ell$ for the 3D fictitious domain problem, $k=1$.}
\label{tab:pc-gridrefFD}
\end{table}

As seen for the interface Poisson problem before, for $\ell\geq 4$ the condition number $\kappa_2(\hat{\MAT{A}})$ behaves like $\sim h^{-2}$ and  
the iteration numbers for the symmetric Gauss-Seidel preconditioner $\MAT{P}_\text{SGS}$ grow approximately like $h^{-1}$.
For the block preconditioners $\MAT{P}_\MAT{A}, \MAT{P}_\MAT{D}, \MAT{P}_\MAT{B}$, we observe almost constant iteration numbers for increasing level $\ell$. For all three block preconditioners the number of iterations roughly doubles when going from the coarsest level $\ell=0$ to the finest one $\ell=6$. 
The influence of the interface position on condition numbers and PCG iteration numbers shows a similar behavior as for the Poisson interface problem. We therefore do not report the numbers here.

\paragraph{Conflicts of interest} This study does not have any conflicts to disclose.

\bibliographystyle{siam}
\bibliography{literatur}


\end{document}